\documentclass[12pt,a4paper]{article} 

\usepackage[utf8]{inputenc} 
\usepackage[T1]{fontenc}

\usepackage{amsmath}
\usepackage{amsthm}
\usepackage{amssymb}
\usepackage{amsfonts}
\usepackage{dsfont} 
\usepackage{url}
\usepackage{graphicx}
\usepackage{hyperref}
\usepackage{color}
\usepackage{tikz}
\usepackage{pgfplots}
\usetikzlibrary{shapes.misc}
\usepackage{enumerate}
\usepackage{fullpage}

\usepackage{epsfig,graphicx,color} 
\definecolor{darkblue}{rgb}{.2, 0.2,.8}
\definecolor{darkgreen}{rgb}{0,0.5,0.3}
\definecolor{darkred}{rgb}{.8, .1,.1}

\theoremstyle{plain}
\newtheorem{theorem}{Theorem}[section]
\newtheorem{lemma}[theorem]{Lemma}
\newtheorem{proposition}[theorem]{Proposition}

\newtheorem{as}{Assumption}

\pgfplotsset{compat=1.13}
\renewcommand{\P}{\mathbb{P}} 
\newcommand{\E}{\mathbb{E}} 
\newcommand{\N}{\mathbb{N}} 
\newcommand{\R}{\mathbb{R}} 
\newcommand{\1}{\mathds{1}} 
\newcommand{\limn}{\lim_{n \to \infty}}

\newcommand{\const}{{\rm const}}

\newcommand{\Addresses}{{%
  \bigskip
  \footnotesize

  P.~Dyszewski \textsc{Instytut Matematyczny Uniwersytetu Wroc\l{}awskiego, Pl. Grunwaldzki 2/4 50-384, 
   Wroc\l{}aw, Poland}\par\nopagebreak
   \textsc{Fakultät für Mathematik, Technische Universität München, 
  Boltzmannstr.~3, 85748 Garching, Germany}\par\nopagebreak
 \textit{E-mail address:} \texttt{piotr.dyszewski@math.uni.wroc.pl}

  \medskip

  N.~Gantert, \textsc{Fakultät für Mathematik, Technische Universität München, 
  Boltzmannstr.~3, 85748 Garching, Germany}\par\nopagebreak
  \textit{E-mail address:} \texttt{gantert@ma.tum.de}
  
  \medskip
  
  T.~Höfelsauer, \textsc{Fakultät für Mathematik, Technische Universität München, 
  Boltzmannstr.~3, 85748 Garching, Germany}\par\nopagebreak

}}

\begin{document}

\title{The maximum of a branching random walk with stretched exponential tails}
\author{Piotr Dyszewski, Nina Gantert and Thomas H\"ofelsauer}

\maketitle

\begin{abstract}
  We study a one-dimensional branching random walk in the case when the step size distribution has a stretched exponential tail, and, in particular, no finite exponential moments. The tail of the step size $X$ decays as $\P[X > t] \sim a \exp\{-\lambda t^r\}$ for some constants $a, \lambda > 0$ where $r \in (0,1)$. We give a detailed description of the asymptotic behaviour of the position of the rightmost particle, proving almost sure limit theorems, convergence in law and some integral tests. The limit theorems reveal interesting differences betweens the two regimes $ r \in (0, 2/3)$ and $ r \in (2/3, 1)$, with yet different limits in the boundary case $r = 2/3$.


\end{abstract}

{\bf Resum\'e} Nous \'etudions une marche al\'eatoire branchante uni-dimensionelle quand les  d\'eplacements n'ont pas des moments exponentiels. Plus pr\'ecisement, la queue d'un d\'eplacement $X$ se comporte comme $\P[X > t] \sim a \exp\{-\lambda t^r\}$ pour des constantes $a, \lambda > 0$ et $r \in (0,1)$. Nous donnons une description d\'etaill\'ee du comportement asymptotique du maximum, en montrant des lois limites presque s\^ures, des theorèmes de convergence en loi et des tests int\'egrals. Ces lois limites diverses font appara\^itre des diff\'erences inter\'essantes entre les deux r\'egimes $ r \in (0, 2/3)$ et $ r \in (2/3, 1)$, et le cas critique $r = 2/3$ est encore diff\'erent.

\section{Introduction} \label{intro}

	We study branching random walk, which is a discrete time Galton-Watson process with a spatial component. Given a reproduction law with expectation $m>1$ and a step size distribution represented by a 
centred random variable $X$ the 
	evolution of the branching random walk can be described as follows.  At time $n=0$ we place one particle at the origin of the real line $\R$. At time $n=1$ this particle splits according to 
	the reproduction law and each new particle performs an independent step, according to the step size distribution. We assume that the branching mechanism and the displacements 
	are independent. The~particles evolve in the same way,
	independently of other particles. We refer to Section~\ref{prel} for a more detailed description of the model. \\
	
	We are interested in the position of the rightmost particle at time $n$, which we will denote by $M_n$. In the case when the step size distribution
        has (some) exponential moments,
	the asymptotic behaviour of $M_n$ is fairly well understood (see the recent monograph~\cite{S15} and references therein).
        We will investigate the case of steps with stretched exponential distribution, when the upper tail of $X$  is of the form
	\begin{equation*}
		\P[X >  x] = a(x) e^{-\lambda x^r}
	\end{equation*} 
	where $a(x) \to a$ as $x \to \infty$ for some constants $\lambda, a>0$ and $r \in (0,1)$. 
The law of large numbers for $M_n$ proved in~\cite{G00} asserts that, under some mild technical conditions, 
almost surely on the set of survival 
	\begin{equation*}
		\lim_{n \to \infty }\frac{M_n}{n^{1/r}} = \alpha:=\left( \frac{\log m}{\lambda} \right)^{1/r}.
	\end{equation*} 
	In the present article we provide a more detailed description of  $M_n$. More precisely we investigate the second term in the asymptotic expansion and prove almost sure convergence of 
	\begin{equation*}
		\frac{M_n -\alpha n^{1/r}}{n^{2-1/r}}
	\end{equation*}
	for $r>\frac 23$ and convergence in law of
	\begin{equation*}
		\frac{M_n-\alpha n^{1/r}}{n^{1/r-1}}
	\end{equation*}
	for $r \leq \frac 23$, see Theorem \ref{thm:second_term}. We also provide a description of upper and lower space-time envelopes of $M_n$ in the latter case. 
	It is well known that the stretched exponential distribution follows the principle of one big jump which we apply to our analysis of $M_n$. The~biggest displacement up to generation $n$ has a leading term $\alpha n^{1/r}$ followed by fluctuations of the 
	order $n^{1/r-1}$. For $r < \frac 23$ the asymptotics of $M_n$ is determined  by the aforementioned one big displacement while the contribution of other particles is  negligible. The case $r >\frac 23$ is slightly different, since one big jump is supplemented 
	by a ``moderate deviations'' contribution of the order $n^{2-1/r}$ coming from other particles. In the boundary case $r =\frac 23$ one sees fluctuations of order $n^{1/r-1} = n^{2-1/r}$ coming from both the behaviour of the biggest jump and 
	other particles. While there has been a lot of recent interest in the case of step distributions with regularly varying tails, see \cite{M16}, \cite{BHSP17}, \cite{BHSP18} and \cite{BMPR19}, it seems that stretched exponential tails have been 
	considered only in \cite{G00, DGH20}.	\\
	
The paper is organized as follows. In Section~\ref{prel} we present the necessary preliminaries concerning the step size distribution and the branching mechanism followed by a detailed description of our model. The main results are presented 
in Section~\ref{sec:main} which also contains some heuristics. The proofs of the main results are in Section~\ref{sec:proofsST}. In the appendix, we give the proof of Lemma \ref{lem:12} and Lemma \ref{sumstat}, two 
results on iid stretched exponential random variables which we did not find in the literature.

\section{Preliminaries}\label{prel}
	Throughout the article we write 
	$f(x) \ll g(x)$ if $f(x) = o(g(x))$ and $f(x) \sim g(x)$ if $\lim_{x \to \infty} f(x) /g(x) =1$.
	We write ``$\const$'' to denote positive constants whose values are of no significance to us. The actual value of ``$\const$'' may change from line to line.
For better readability, we often omit integer parts when no confusion arises.
	As mentioned in the first paragraph of the introduction, we suppose that the branching mechanism and the displacements are independent. Therefore we can introduce them separately.  

\subsection{Step size distribution}

	Let $X, X_1, X_2, \ldots$ be a collection of iid random variables of zero mean and unit variance and let $S = (S_n)_{n \geq 0}$ be the corresponding random walk, that is $S_0=0$, 
	$S_n = \sum_{k=1}^n X_k$. 
	Throughout the analysis of the branching random walk the behaviour of the probabilities  
	\begin{equation*}
		\P[S_n > x_n], \quad x_n \to \infty
	\end{equation*}
	as $n \to \infty$ plays a crucial role. In the case when Cram\'er's condition holds, that is
	\begin{equation}\label{eq:cramer}
		\E \left[ e^{s|X|} \right]<\infty \quad \mbox{for some }s>0 
	\end{equation}
	it is well known that
	\begin{equation*}
		\log\P[S_n>x_n]  \sim \left\{ \begin{array}{cc} -I(\rho)n , &  x_n = \rho n, \: \rho > 0 \\ 		-\frac{x_n^2}{2n}, &   x_n \ll n^{2/3}. \end{array} \right. 
	\end{equation*}
	where $I(\rho) = \sup_{s\in \R} \left(s\rho - \log\E\left[e^{sX}\right]\right)$, see~\cite{DZ10} for the case $ x_n = \rho n$ if \eqref{eq:cramer} holds and~\cite{modf18} for a complete description with a full  range of possible orders of $x_n$. 
	If on the other hand, $\E \left[e^{s|X|}\right]=\infty$ for any $s>0$ it is known that the probabilities $\P\left[S_n>xn\right]$ decay slower than exponentially in $n$  with the exact rate
	 being determined by the tail $\P[X>x]$ as $x \to \infty$. We will focus on the case of stretched exponential distributions.

	\begin{as} \label{assumption_RW}
		The random variable $X$ is centred $(\E [ X ] =0)$, has variance $1$ $(\E\left[X^2 \right]=1)$ and has a stretched exponential upper tail, that is there exist $\lambda >0$, $r \in (0,1)$ and a~function $a(x)$ with $a(x) \to a$ for $x \to \infty$ 
		such that
		\begin{equation*}
                  \P[X > x]=a(x) e^{-\lambda x^r}
		\end{equation*}
		for all $x\geq 0$.
Furthermore we assume that the lower tail of $X$  satisfies:
\begin{equation}\label{momentass}
\begin{array}{cc}
	\mbox{if $r> 2/3$,} & \lim_{x \to \infty}x^{-\frac{3r-2}{2r-1}} \log\P\left[X<-x\right] = -\infty \\
	\mbox{if $r\leq 2/3$,} & \E\left[|X^k|\right]<\infty,  \forall   k \in \N.\\
\end{array}
\end{equation}
	\end{as}

Note that $0 < (3r-2) /(2r-1)< r $ if $r > \frac 23$. Deviations for a random walk in the case when Cram\'er's condition is not fulfilled go back to~\cite{N69}. The statements we will need are collected in the following lemma.

\begin{lemma} \label{lemma_RW}
          Let Assumption~\ref{assumption_RW} be in force. Then for any constant $c>0$,
          \begin{equation*}
          	\begin{array}{cc}
          		\mbox{if $r> 2/3$,} &  \log \P\left[S_n > c n^{2 -\frac 1r}\right] \sim -\frac{c^2}{2}n^{3-\frac 2r}\\
          		\mbox{if $r\leq 2/3$,} &  \log \P\left[|S_n| > c \sqrt{n\log n}\right] \sim -\frac{c^2}{2}\log n 
          	\end{array}
          \end{equation*}
		If $x_n \gg n^{s}$, where $ s = \frac{1}{2-2r}$, then
		\begin{equation*}
			\P\left[ S_n > x_n \right] \sim n \P\left[ X>x_n \right].
		\end{equation*}
	\end{lemma}

The proof of the first part of this lemma can be found in \cite{EichLoe} and the second follows from Theorem 8.2 in ~\cite{DDS08} .
	
\subsection{Branching mechanism}
	Let $Z = (Z_n)_{n \geq 0}$ be a Galton-Watson process with $Z_0 =1$ and the reproduction law $ (p_k)_{k\geq 0}$.   The key parameter describing the asymptotic behaviour of 
	$Z$ is the mean of the reproduction law denoted by
	\begin{equation*}
		m: = \sum_{k = 0}^\infty k p_k.
	\end{equation*}
	It is well-known that, provided $p_1<1$, the branching process survives with positive probability if and only if $m >1$. In this case one can introduce the probability 
	\begin{equation*}
		\P^*[\: \cdot \: ] = \P[\: \cdot \:  |  \:  \forall \: n \in\N, \: Z_n>0  ].
	\end{equation*}
	The asymptotic growth rate of $Z_n$ will be of crucial importance.
	It can be described by considering the sequence $W_n = m^{-n}Z_n$ which is a non-negative martingale with respect to $\mathcal{F}_n = \sigma( Z_k \: : \: k \leq n)$ and thus has an almost sure limit 
	\begin{equation}\label{eq:2:W}
		W = \lim_{n \to \infty} m^{-n}Z_n.
	\end{equation} 
	The Kesten-Stigum Theorem provides a necessary and sufficient criterion for $W$  to be non-degenerate. 
	\begin{lemma}
		Assume that $m>1$. Then
		\begin{equation*}
			\P^*[W>0] = 1 \quad \Leftrightarrow \quad \E \left[ Z_1 \log^+Z_1 \right] <\infty.
		\end{equation*}
	\end{lemma}	

	The proof can be found in \cite[Chapter 2]{S15}. We will prove our main result in the case when $W>0$ $\P^*$-a.s.
        Our standing assumption on the branching process will be the following.
	
	\begin{as} \label{as:BP} 
          The Galton-Watson process $Z$ is supercritical, that is $m>1$, and we have $\E[ Z_1 \log^+ Z_1] <\infty$.
	\end{as}

\subsection{Branching random walk}
	
	The branching random walk is a discrete time stochastic process that can be described in the following way.
	At~time $n=0$ one particle is placed at the origin of the real line. 
	This particle will start a~population which will be described by the branching process $Z = (Z_n)_{n\geq 0}$. At time $n=1$ the initial particle splits into $Z_1$ new particles which 
	move independently of each other and of $Z_1$. We assume that
	all displacements of particles from their place of birth are independent copies of $X$. Each particle evolves according to this rules independently of all other particles. More precisely, at time $n=2$, 
	each particle, independently of the others, splits into a random number of particles distributed according to the reproduction law. The total number of particles present at the system at time $n=2$ 
	is denoted by $Z_2$. Each particle performs, independently of all other particles and of $Z_1, Z_2$, a step 
which has the same law as $X$. The system continues according to these rules.  
	Let $\mathcal{T} = (V, E)$ be the associated Galton-Watson tree with the initial particle denoted as the root $o \in V$ (see~\cite{S15} for more information and many results on this model). 
	Let $D_n \subset V$ denote the 
	set of particles present in the system at time $n$. Clearly $|D_n | = Z_n$. 
	For $v, w \in V$ write $[v,w]$ for the set of vertices along the unique path in the graph $\mathcal{T}$ from $v$ to $w$ (including $v$ and $w$). Write $|x|=n$ if $x \in D_n$ and $|x|\leq n$ if $x \in \bigcup_{k=0}^nD_k$. For $x,y \in \mathcal{T}$ denote by $x \wedge y$ the last common ancestor of $x$ and $y$. Finally we write $x \leq y$ if $x \in [o ,y]$, that is if $x$ is an ancestor of $y$.
	To model the displacements, assume that each vertex of the tree $\mathcal{T}$, except the root, is labelled 
	with an independent copy of $X$, that is we are given a collection 
	$\{ X_v\}_{v\in V\setminus \{o\}}$ of iid random variables distributed as $X$. The random variable $X_v$ describes the displacement that the particle $v$ took from its birthplace. We set $X_o =0$.
        Then the position of the particle $v$ is equal to
	\begin{equation*}
		S_v = \sum_{u \in [o,v]} X_u
	\end{equation*} 
	and the position of the rightmost particle at time $n$ is
	\begin{equation*}
		M_n = \max_{|v|=n}S_v.
	\end{equation*}
	It is well known, that if~\eqref{eq:cramer} is satisfied, then $M_n$ has a linear speed, that is $n^{-1}M_n$ converges to a constant a.s. (see \cite{B76, H74, K75}) and the second term is of logarithmic order. More precisely, denote $\varphi(s) = \log m + \log \E \left[e^{sX} \right]$ and
	suppose that there exists $s_0>0$ such that $s_0\varphi'(s_0) = \varphi(s_0)$. Then, under some mild technical assumptions, 
	\begin{equation*}
		\frac{M_n + \varphi'(s_0) n}{\log n} \overset{\P}\to \frac{3}{2s_0}. 
	\end{equation*}
	see~\cite{HS09, S15}. Moreover it is known that $M_n + \varphi'(s_0) n - 3\log (n)/(2s_0)  $ converges in distribution~\cite{A13, S15}.

	In our case, as proved in~\cite{G00}, under Assumptions~\ref{assumption_RW} and \ref{as:BP}, $M_n$ grows faster than linear in $n$.
	\begin{lemma}
		Let Assumptions~\ref{assumption_RW} and \ref{as:BP} be in force. Then 
		\begin{equation}\label{linspeed}
			\limn \frac{M_n}{n^{1/r}} = \alpha = \Bigl( \frac{\log m}{\lambda} \Bigr)^{1/r} \quad \P^*\text{-a.s.}
		\end{equation}
	\end{lemma}

\section{Main results}\label{sec:main}

	We can now present our main results. Denote
	\begin{equation}\label{sigmarhodef}
		\sigma = \frac{\alpha^{1-r}}{\lambda r}, \qquad \rho = \sum_{k=0}^\infty m^{-k}\P\left[Z_k>0 \right].
	\end{equation}

	\begin{theorem}\label{thm:second_term}
		Suppose that Assumptions \ref{assumption_RW}  and \ref{as:BP} are satisfied.
		If $r \in (\frac{2}{3},1)$ then
		\begin{equation}\label{mainlarger}
                  \limn \frac{M_n-\alpha n^{1/r}}{n^{2-1/r}}  =
                  \frac{r \log m}{2 \alpha}  \quad \P^*\text{-}a.s.
		\end{equation}
		If $ r \in (0, \frac 23)$ then
			\begin{equation}\label{convdistrsmallr}
			\frac{M_n-\alpha n^{1/r}}{\sigma n^{1/r-1}} \overset{d}\to V,
                \end{equation}
where $V$ is a random variable with c.d.f.
                \begin{equation}\label{mainsmallr}
                \P^*[V\leq x] = H(x)= \E^* \left[\exp \left\{- a \rho We^{-x} \right\} \right].
		\end{equation}
		If $r =\frac 23$ then
		\begin{equation*}
			\frac{M_n - \alpha n^{3/2}}{\sigma\sqrt{n}} \overset{d}\to  V_{2/3},
                \end{equation*}
where $ V_{2/3}$ is a random variable with c.d.f.
                \begin{equation}\label{mainsmallrcrit}
\P^*[V_{2/3}\leq x] = H_{2/3}(x) :=  \E^* \left[ \exp \left\{-a \rho W e^{-x + \sigma^{-2}/2}  \right\}\right].
		\end{equation}
	\end{theorem}

 	Theorem \ref{thm:second_term} states for $r \leq \frac 23$ a convergence in distribution. It is natural to ask about the almost sure behaviour of $M_n$ in this case.
        \begin{theorem}\label{thm:second_LIL}
		Suppose that Assumptions \ref{assumption_RW}  and \ref{as:BP} are satisfied. If $ r \in (0, \frac 23)$ we have $\P^*$ - a.s.
		\begin{equation*}
                  \liminf_{n \to \infty}\frac{M_n - \alpha n^{1/r} + \sigma n^{1/r-1}\log\log n}{n^{1/r-1}} =
                  \sigma \log\left(a \rho W\right),
		\end{equation*}
		and for any positive, non-decreasing function $\psi \colon [1, \infty) \to \R$ such that $\psi(n) = o(n)$,
		\begin{equation*}
                  \limsup_{n \to \infty}\frac{M_n - \alpha n^{1/r} - \sigma n^{1/r-1}\psi(n)}{n^{1/r-1}} =
                  \begin{cases} -\infty & \text{ if } \int\limits_1^\infty e^{-\psi(x)} \: dx <\infty\\
                    +\infty & \text{ if } \int\limits_1^\infty e^{-\psi(x)} \: dx  = \infty\, .\\
                    \end{cases}
                  \end{equation*}
                  
                  If $r =\frac 23$ then $\P^*$- a.s.  
		\begin{equation}\label{limsupandinf} 
			\limsup_{n \to \infty} \frac{M_n-\alpha n^{3/2}}{ \sqrt{n} \log n} = \sigma
	\quad \mbox{and}\quad	-\infty < \liminf_{n \to \infty } \frac{M_n -\alpha n^{3/2}}{ \sqrt{n\log n }} 
< \infty.
		\end{equation}
              \end{theorem}

              After presenting the main results, we describe the strategy of the proofs.
              First, we explain the arguments concerning almost sure convergence and convergence in law in Theorem~\ref{thm:second_term}.
              Then, we give the arguments leading to a description of the upper and lower space-time envelopes in Theorem \ref{thm:second_LIL}.

\subsection{Almost sure and weak convergence}
             
	In order to understand the limiting distributions in the case $r \leq \frac 23$ and to illustrate what leads to this behaviour of $M_n$, we first introduce a simpler process which we use in the proof of Theorem~\ref{thm:second_term}. 
	Consider the biggest displacement of particles which have (at least) one descendant at generation $n$, i.e.
	\begin{equation*}
		N_n = \max \left\{ X_v \: : \: v \in \mathcal{N}_n \right\}, \quad \mathcal{N}_n =\left\{ v \in \bigcup _{k=1}^n D_k \: : \: \exists x \in D_n , \: v \leq x \right\}.
	\end{equation*}
	Due to Assumption~\ref{assumption_RW}, the law of the displacements lies in the maximum domain of attraction of the Gumbel law.
	Since $N_n$ is just a maximum of 
	\begin{equation*}
		Y_n=   |\mathcal{N}_n|
	\end{equation*}	 
	independent random variables it is relatively easy to obtain its asymptotic behaviour. In what follows we describe the behaviour of the extremes of $\left\{ X_v \: : \: v \in \mathcal{N}_n \right\}$ using the convergence of point processes, 
	that is measurable functions taking values in the space of point measures equipped with the vague topology~\cite{R87}.
	The convergence mentioned in Proposition~\ref{prop:Nn} below is the convergence in distribution with respect to vague convergence of measures on $\R$. Equivalently by \cite[Proposition 3.19]{R87} the point process $\Lambda_n$ converges in distribution to a point process $\Lambda$ if and only if 
	for any continuous, non-negative $f\colon \R \to \R $ with compact support
	\begin{equation*}
		\int f(s) \: \Lambda_n(ds)  \overset{d}\to \int f(s) \: \Lambda(ds).
	\end{equation*}
	In the sequel we will use a special class of random measures. A point process $\Lambda$ is a Poisson point process with intensity measure $\mu$ if and only if for any $f \colon \R \to \R$ continuous, non-negative with compact support,
	\begin{equation}\label{laplace}
		\E \left[ \exp \left\{ - \int f(s) \Lambda(ds) \right\} \right] = \exp \left\{ -\int_\R \left( 1-e^{-f(s)} \right) \mu(ds) \right\},
	\end{equation}
	see~\cite[Proposition 3.6]{R87}.
	We will denote by $\epsilon_x$, for $x \in \R$ the probability measure concentrated at $x$. That is $\epsilon_x(A) =1$ if $x \in A$ and $\epsilon_x(A) =0$ otherwise. We refer to~\cite{K83, R87} for an introduction to the topic of random measures.  \\ 
	
	\begin{proposition}\label{prop:Nn}
		Suppose that Assumptions \ref{assumption_RW}  and \ref{as:BP} are satisfied. Then  
		\begin{equation*}
			\frac{N_n-\alpha n^{1/r}}{\sigma n^{1/r-1}} \overset{d}\to V 
\end{equation*}
where $V$ has the c.d.f. given by \eqref{mainsmallr}.
		Moreover the point process on $\R$ given by
		\begin{equation*}
			\Lambda_n = \sum_{v \in \mathcal{N}_n} \epsilon_{\bar{X}_v}, \quad \bar{X}_v = \frac{X_v-\alpha n^{1/r}}{\sigma n^{1/r-1}}
		\end{equation*}
		converges in distribution to a random measure $\Lambda$, which conditioned on $W$ is a Poisson point process with intensity $\mu$ given by
		\begin{equation}\label{intensity}
			\mu (W, dx) = a \rho W e^{-x}dx.
                      \end{equation}
	\end{proposition}

	We can already see, that the asymptotics of $M_n$ and $N_n$ coincide for $ r < \frac 23$.  In fact, we will prove the following.
        \begin{lemma}\label{comparelemma}
          Let the Assumptions \ref{assumption_RW}  and \ref{as:BP} be in force. For $ r < \frac 23$, 
	\begin{equation}\label{compare}
		\lim_{n \to \infty }\frac{M_n - N_n}{ n^{1/r-1}} =0,\,  \P^*-a.s.
              \end{equation}
              \end{lemma}
	The scaling and convergence of $M_n$ given in \eqref{convdistrsmallr}
        for
        $  r <\frac 23$ is a direct consequence of Proposition~\ref{prop:Nn} and \eqref{compare}.
        It says that
              $M_n$ is asymptotically determined by one big displacement. \\

              The boundary  case  $  r =\frac 23$ is more subtle and requires more detailed information about the extremes of the displacements.
              Let us give a heuristic argument for \eqref{mainsmallrcrit}.
              Consider the order statistics of $\left\{ X_v \: : \: v \in \mathcal{N}_n \right\}$,
	\begin{equation}\label{eq:3:extremes}
		N_n = N_n^{(1)} \geq N_n^{(2)} \geq \ldots \geq N_n^{(Y_n)} = \min_{v \in \mathcal{N}_n}X_v.
	\end{equation} 
	It turns out that when $r =\frac 23$ there is a polynomial number of big jumps in $\{X_{v}\}_{|v|\leq n}$ that can affect $M_n$. Consider a particle $v \in D_n$ that had an ancestor whose displacement is among the aforementioned big jumps, say $N_n^{(j)}$ for $j(v)=j \leq n^\const$. Then the position of $v$
	is composed of   $N_n^{(j)}$ and a sum $S_{n,j}=S_{n,j}(v)$ of displacements of other ancestors of $v$. 
	One can show that given $Z$ and $\{N_n^{(j)}\}_{j \leq n^\const}$, the $S_{n,j}$'s are asymptotically independent.
	Since the $S_{n,j}$'s are also asymptotically normal, by conditioning on $Z$ and the $N_n^{(j)}$'s, we see that
	\begin{align*}
		\P^* \left[ \frac{M_n - \alpha n^{3/2}}{\sigma\sqrt{n}} \leq x \right] & \approx \P^* \left[ \max_{j \leq n^\const}\frac{N_n^{(j)} + S_{n,j} - \alpha n^{3/2}}{\sigma\sqrt{n}} \leq x \right] \\
					& \approx \E^* \left[ \prod_{j \leq n^\const} \Phi \left( x - \frac{N_n^{(j)} - \alpha n^{3/2}}{\sigma\sqrt{n}} \right) \right] \\
					& \approx \E^* \left[ \prod_{v \in \mathcal{N}_n} \Phi \left( x - \frac{X_v - \alpha n^{3/2}}{\sigma\sqrt{n}} \right) \right],
	\end{align*}
	where $\Phi$ denotes the cumulative distribution function (c.d.f.) of a centred Gaussian distribution with variance $\sigma^{-2}$.	The last quantity can be described in terms of the point process $\Lambda_n$:
	\begin{align*}
		\E^* \left[ \prod_{v \in \mathcal{N}_n} \Phi \left( x - \frac{X_v - \alpha n^{3/2}}{\sigma\sqrt{n}} \right) \right] & = \E^* \left[ \exp \left\{  \sum_{v \in \mathcal{N}_n}   \log \left( \Phi \left( x - \frac{X_v - \alpha n^{3/2}}{\sigma\sqrt{n}} \right)\right) \right\} \right] \\
			& =  \E^* \left[ \exp \left\{  \int  \log \left( \Phi \left( x - y \right)\right) \Lambda_n(dy) \right\} \right] \\
			& \to \E^* \left[ \exp \left\{  \int  \log \left( \Phi \left( x - y \right)\right) \Lambda(dy) \right\} \right] \\
			& =\E^* \left[\exp \left\{ - \int \left( 1-\Phi ( x - y) \right) \mu(W, dy) \right\} \right] \\
& =\E^* \left[ \exp \left\{- a \rho W \int \Phi(y-x)  e^{-y} dy\right\}\right] =  H_{2/3}(x),
	\end{align*}
where we used \eqref{laplace} in the second to last equality. We conclude that both large and typical displacements of the particles contribute to the second term in the asymptotic expansion of $M_n$. 
	From this sketch one can also see that $H_{2/3}(x)$ is the c.d.f. of the rightmost particle of the point process $\Lambda$ with independent, Gaussian shifts. More precisely, let $\{\xi_k\}_{k \geq 1}$ be the points of a point process with intensity measure given by \eqref{intensity} such that
	\begin{equation*}
		\Lambda = \sum_{k=1}^\infty \epsilon_{\xi_k}   
	\end{equation*}
	and take a collection $\{\eta_k\}_{k\geq 1}$ of iid random variables with common cumulative distribution function $\Phi$, independent of $\Lambda$. Consider a new point process $\sum_{k=1}^\infty \epsilon_{\xi_k+\eta_k}$ and note that, by the same arguments as above, the distribution of the rightmost particle is given by
	\begin{align*}
		\P^* \left[ \max_{k \geq 1} \left( \xi_k +\eta_k \right) \leq x\right]  & = \E^* \left[ \exp \left\{ \int \log \Phi(x-y) \: \Lambda (dy) \right\} \right] \\
			& = \E^*\left[ \exp \left\{  -\int  \left(1-\Phi ( x - y) \right) \mu(W, dy) \right\} \right] = H_{2/3}(x).
	\end{align*}

	In the case $r > \frac 23$ the limiting behaviour is different. In contrast to the boundary case, there is an exponential number of $N_n^{(j)}$'s, i.e. big jumps that can affect $M_n$. This in turn leads to a much greater number of $S_{n,j}$'s that can contribute which in turn yields a more concentrated asymptotic behaviour. 

\subsection{The space-time envelopes}

 We already mentioned the significance of the biggest displacement for the convergence in law. As we will see, this is also the case for the 
	almost sure behaviour. 

		\begin{proposition}\label{prop:Nn_LIL}
		Suppose that Assumptions \ref{assumption_RW}  and \ref{as:BP} are satisfied. Then 
		\begin{equation*}
			\liminf_{n \to \infty}\frac{N_n - \alpha n^{1/r} + \sigma n^{1/r-1}\log\log n}{n^{1/r-1}} = \sigma \log\left(a \rho W\right), \, \P^*\text{-a.s.}
		\end{equation*}
		and for any positive, non-decreasing function $\psi \colon [1, \infty) \to \R$ such that $\psi(n) = o(n)$, 
		\begin{equation}\label{Nnas}
                  \limsup_{n \to \infty}\frac{N_n - \alpha n^{1/r} - \sigma n^{1/r-1}\psi(n)}{n^{1/r-1}} =
                  \begin{cases} -\infty & \text{ if }	\int\limits_1^\infty e^{-\psi(x)} \: dx <\infty \\
                    +\infty & \text{ if } \int\limits_1^\infty e^{-\psi(x)} \: dx =\infty\, . \\
                    \end{cases}
		\end{equation}
	
	\end{proposition}
	
	Since $M_n - N_n = o\left(n^{1/r-1}\right)$ for $r < \frac 23$, see \eqref{compare},
        we see that in this case the description of $M_n$ will be exactly the same. The boundary case $r =\frac 23$ is more subtle. We already mentioned, in the heuristics behind the proof of Theorem~\ref{thm:second_term}, that $M_n$ is composed of the biggest jumps $(N_n^{(j)})_{j \leq n^{\rm const}}$ and sums of typical displacements $(S_{n,j})_{j\leq n^{\rm const}}$ (recall~\eqref{eq:3:extremes} and 
the discussion that follows).
        Since with high probability the $S_{j,n}$'s are in $[-\sqrt{n\log n}, \sqrt{n\log n}]$ 
 and $N_n^{(j)} - \alpha n^{3/2}$ is in $[-\sigma\sqrt{n}\log\log n, \sigma\sqrt{n}\log n]$ with high probability, one can deduce the correct order of $M_n$ by comparing both intervals. It turns out that the upper time space envelope of $M_n$ is determined by the upper 
        space-time envelope of the biggest displacement and the lower space-time envelope is determined by the 
        sum of typical displacements. 

\section{Proofs}\label{sec:proofsST}
	We begin with some auxiliary lemmas followed by the proof of Proposition~\ref{prop:Nn}. Next we present the arguments for our main result.
	
\subsection{Some auxiliary results}

	Recall that $W_n = m^{-n}Z_n$ is a positive martingale whose limit $W$ plays a significant role in the asymptotics of our model. For technical reasons we need almost sure bounds for $W_n$. 

	\begin{lemma} \label{lemma_W_tail}
		Let Assumption \ref{as:BP} be in force. There exists $\beta>0$ such that 
		\begin{equation*}
			\sum_{n=0}^\infty \P^* \bigl[W_n \leq n^{-\beta} \bigr] < \infty \quad \text{and} \quad \sum_{n=0}^\infty \P^* \bigl[W_n > n^\beta \bigr] < \infty.
		\end{equation*}
	\end{lemma}	

	\begin{proof}
		The second part is immediate, since $\P^* \bigl[W_n > n^\beta \bigr] \leq \E^* [W_n] n^{-\beta} = n^{-\beta}$.
		For the first part we need to distinguish between Schröder and Böttcher cases, that is $p_0+p_1>0$ and $p_0+p_1=0$ respectively. 
		In the former case, by~\cite[Theorem 4]{FW06} (note that $\P[W>0] = \P[Z_n >0, \: n\geq 0]$ under Assumption \ref{as:BP}),
		\begin{align*}
			\P^*\left[W_n \leq n^{-\beta}\right] & = \P^*\left[Z_n \leq m^nn^{-\beta}\right] \\
				&\leq \const \cdot \P\left[0< Z_n \leq m^nn^{-\beta}\right] \sim \const \cdot \P\left[0<W < n^{-\beta}\right].
		\end{align*}
		By~\cite[Theorem 4]{BB93} the left tail of $W$, i.e. $\P^*[W \leq x ]$ exhibits a polynomial decay, so for $\beta>0$ large enough $\sum_n\P^*[W \leq n^{-\beta}]< \infty$. 
		Turning to the Böttcher case we denote $k^* = \min \{ k \: : \: p_k>0\}$. One can use~\cite[Theorem~6]{FW06}, for $\varepsilon = (\log m/ k^*)^{-1}$, to get
		\begin{equation*}
			\limsup_{n \to \infty} (2n)^{-\varepsilon \beta} \log \P^*[Z_n \leq m^nn^{-\beta}] <0
		\end{equation*}
		and so the probabilities $\P^*[W_n \leq n^{-\beta}]$ decay faster than any polynomial for any fixed $\beta>0$.
	\end{proof}
	
	Lemma~\ref{lemma_W_tail} implies that for sufficiently large $n$, $n^{-\beta} \leq W_n \leq n^\beta$, $\P^*$-a.s. 
	The next two lemmata are statements about iid stretched exponential random variables which we did not find in the literature. We provide the proofs in the appendix.

\begin{lemma}\label{lem:12}
		Let Assumption \ref{assumption_RW} be in force. Let $\delta \in \left(\alpha 2^{-1/r}, \alpha\right)$ and take $x_n$ to be any sequence such that $x_n \sim \alpha^{r-1}n^{1-1/r} $. Then for $\hat{X} =  X \1_{\{ X < \delta n^{1/r}\}}$ we have
		\begin{equation*}
			 \E \left[\exp\left\{ \lambda x_n \hat X \right\} \right] \leq 1 + \frac{\lambda^2  x_n^2}{ 2} + o\left(\frac{1}{  n^{2(1/r -1)}}\right).
		\end{equation*}
	\end{lemma}

\begin{lemma}\label{sumstat}
		Let Assumption \ref{assumption_RW} be in force and $r > \frac 23$. Then for $m > 1$ and $\varepsilon > 0$,
\begin{equation*}
	\sum_{n=1}^\infty \P\Bigl[S_n - \alpha n^{1/r} \geq \frac{(1+\varepsilon) r \log m}{2 \alpha} n^{2-1/r} \Bigr] m^n
< \infty\, .
\end{equation*}
	\end{lemma}

	We will often use the following asymptotics for the $r$-th power, which follows easily from the mean value theorem. Assume that $(a_n), (b_n)$ are positive sequences such that $a_n \to \infty, b_n \to \infty, \frac{b_n}{a_n} \to 0$. Then
        \begin{equation}\label{rpower}
          (a_n + b_n)^r = a_n^r + b_n \frac{r}{a_n^{1-r}} + o \left( \frac{b_n}{a_{n}^{1-r}} \right).
          \end{equation}
\subsection{The biggest displacement}	
	
	\begin{proof}[Proof of Proposition~\ref{prop:Nn}]
		Recall $Y_n= | \{ v \: : \: \exists x \in D_n, \: v \leq x \}|$. Following~\cite{Durrett1983} write
		\begin{equation*}
			Y_n =\sum_{j=1}^n Z_n^{(j)}, \qquad Z_n^{(j)} =  | \{ v \in D_j \: : \: \exists x \in D_n, \: v \leq x \}|.
		\end{equation*}
		Note that by the branching property and the law of large numbers, for fixed $k \in \N$ as $n \to \infty$,  $Z_n^{(n-k)} Z_{n-k}^{-1} \to \P[Z_k>0]$ $\P^*$-a.s. and therefore by an appeal to~\eqref{eq:2:W}, 
		\begin{equation*}
			m^{-n}Y_n \to \rho W \qquad \P^* - a.s.
		\end{equation*}
		To prove weak convergence of $N_n$, fix $x \in \R$, take
		\begin{equation*}
                  e_n = e_n(x) := \alpha n^{1/r} + \sigma
                  n^{1/r-1}x
		\end{equation*}
		and write
		\begin{equation*}
			\P^* \left[\left. \frac{N_n - \alpha n^{1/r}}{\sigma n^{1/r-1}} \leq x \: \right| \:  \mathcal{T} \: \right] =\P[N_n \leq e_n \: | \: \mathcal{T} \:] = (1-\P[X>e_n])^{Y_n}.
		\end{equation*}
		Since, using \eqref{rpower},  $m^n\P[X>e_n] \to ae^{-x}$ and $m^{-n}Y_n \to \rho W$ $\P^*$-a.s. we have
		\begin{equation*}
			\P^*[N_n \leq e_n \: | \: \mathcal{T} \: ] = \exp \left\{ -a \rho W e^{-x}(\Delta_n +1) \right\}
		\end{equation*}
		for some $\Delta_n \to 0$, $\P^*$ -a.s. The weak convergence of $N_n$ follows after taking expectations and using the dominated convergence theorem. 
		
		We now turn our attention to  the random measures $\Lambda_n$, which we will analyse using the corresponding Laplace transforms. Take $f$ from the class $C_c^+(\R)$ of continuous, non-negative, compactly supported functions and note that 
		by a standard approximation argument
		\begin{equation}\label{eq:4:wer}
			\lim_{n \to \infty}m^n \E \left[ f \left( \frac{X-\alpha n^{1/r}}{\sigma n^{1/r-1}} \right) \right] = a \int f(s) e^{-s}ds.
		\end{equation}
		In other words the sequence of measures $m^n \P[ \sigma^{-1} n^{1-1/r} (X-\alpha n^{1/r}) \in \cdot ]$ converges to $a e^{-s}ds$ in vague topology. This, by the merit of~\cite[Proposition 3.21]{R87}, implies
		that $\Lambda_n \to^d \Lambda$. 
		For convince we will sketch the argument. The convergence~\eqref{eq:4:wer} for $f \in C_c^+(\R)$ implies, by boundedness of $f$, that
		\begin{equation*}
			\lim_{n \to \infty}m^n L_n(f)  =  -a \int (1-e^{-f(s)} )e^{-s}ds, \, \text{ where } L_n(f) = \log \E \left[\exp \left\{ -  f \left( \frac{X-\alpha n^{1/r}}{\sigma n^{1/r-1}} \right) \right\} \right]
		\end{equation*}
		which further allows us to infer that $\P^*$-a.s. for any $t>0$, 
		\begin{align*}
			\E \left[ \left. \exp \left\{ - t \int f(s) \: \Lambda_n (ds) \right\} \right| \mathcal{T} \right] & = \exp \left\{  Y_nt L_n(f) \right\} \to \exp \left\{ -a \rho W t\int (1-e^{-f(s) })e^{-s}ds \right\} .
		\end{align*}
		If we combine the above convergence with the dominated convergence theorem we can conclude that $\Lambda_n \to^d \Lambda$.

	\end{proof}

	\begin{proof}[Proof of Proposition~\ref{prop:Nn_LIL}]
		To treat the lower space-time envelope take
		$\gamma >0 $ and define 
		\begin{equation*}
			f_n = f_n(\gamma, Y_n) := \alpha n^{1/r} - \sigma n^{1/r-1}\log\log n + \sigma n^{1/r-1}\log\left(\gamma a \frac{Y_n}{m^n} \right).
		\end{equation*}
		Using the inequality $1-x \leq e^{-x}$ and \eqref{rpower}, a calculation gives 
		\begin{align*}
			\P^*[ N_n \leq f_n| \: \mathcal{T}\: ] & \leq  \exp \left( - Y_n \P^*[X> f_n \: | \: \mathcal{T} \:] \right) \\
			 & =\exp \left(  -\gamma^{-1} (\log n) \exp \left\{  1+\Delta_n  \right\} \right)
		\end{align*}
		for some $\Delta_n \to 0$ $\P^*$-a.s.
		After taking expectations we see that $\sum_{n=1}^\infty \P^*[ N_n \leq f_n] <\infty$ provided $\gamma <1$.  By the Borel-Cantelli lemma  
		\begin{equation*}
			\frac{N_n - \alpha n^{1/r-1} + \sigma n^{1/r-1}\log\log n}{n^{1/r-1}} > \sigma\log\left(\gamma a \frac{Y_n}{m^n}\right)
		\end{equation*}
		for sufficiently large $n$. Letting $n \to \infty$ followed by $\gamma \uparrow 1$ yields
		\begin{equation}\label{llupp}
			\liminf_{n \to \infty}\frac{N_n - \alpha n^{1/r-1} + \sigma n^{1/r-1}\log\log n}{n^{1/r-1}} \geq \sigma \log\left(a \rho W\right).
		\end{equation}
		To show that ``$\leq$'' holds in \eqref{llupp} as well, fix $\gamma >1$, take $n_k = k^{1+\varepsilon}$ for $\varepsilon <\gamma-1$ and consider the $\sigma$-algebras
		\begin{equation*}
			\mathcal{I}_k = \sigma \left( \mathcal{T}, \:X_{v} \: : \:  |v| \leq n_k\right).
		\end{equation*} 
		We have
		\begin{equation*}
			\P^*[N_{n_k} \leq f_{n_k} \: | \: \mathcal{I}_{k-1}] = \1_{\left\{ N_{n_{k-1}}  \leq f_{n_k}\right\}} \P^*\left[ \left. \max_{n_{k-1} < |v| \leq n_k} X_v \leq f_{n_k} \right| \mathcal{I}_{k-1} \right].
		\end{equation*}
		We first show that the events
		\begin{equation*}
			A_k = \left\{ N_{n_{k-1}}  > f_{n_k}\right\}
		\end{equation*}
		can occur only finitely many times. We have  
		\begin{align*}
			\P^*[A_k]   \leq& \P^*\left[  N_{n_{k-1}} > \alpha n_{k}^{1/r} - \sigma n_k^{1/r-1}\log\log n_k + \sigma n_k^{1/r-1}\log\left(\gamma a \frac{Z_{n_k}}{m^{n_k}} \right) \right] \\  
			\leq & \P^*[W_{n_k} \leq n_k^{-\beta}] \\
				& + \P^*\left[ W_{n_k} > n_k^{-\beta} , \: N_{n_{k-1}} > \alpha n_{k}^{1/r} - \sigma n_k^{1/r-1}\log\log n_k + \sigma n_k^{1/r-1}\log\left(\gamma a \frac{Z_{n_k}}{m^{n_k}} \right)  \right] \\
			\leq & \P^*[W_{n_k} \leq n_k^{-\beta}] \\
				& + \P^*\left[ N_{n_{k-1}} >\alpha n_k^{1/r} - \sigma n_k^{1/r-1}\log\log n_k - \beta \sigma n_k^{1/r-1}\log n_k + \sigma n_k^{1/r-1}\log(\gamma a) \right] \\
			\leq & \P^*[W_{n_k} \leq n_k^{-\beta}] \\& + \const \cdot  m^{n_{k-1}} \P^*\left[X > \alpha n_k^{1/r} - \sigma n_k^{1/r-1}\log\log n_k - \beta \sigma n_k^{1/r-1}\log(\gamma n_k) 
    \right] \\
				& \leq \P^*[W_{n_k} \leq {n_k}^{-\beta}]  + \const \cdot m^{- k^{\varepsilon}}.
		\end{align*} 
		Applying the Borel-Cantelli lemma and using Lemma \ref{lemma_W_tail}, $\1_{A_k^c} =1$ for sufficiently large $k$. Set 
		\begin{equation*}
			\Delta Y_{n_k} = \left| \Delta \mathcal{N}_{n_k} \right|, \quad \Delta \mathcal{N}_{n_k} =  \left\{ v \in \bigcup_{j=n_{k-1}}^{n_k} D_j \: : \: \exists w \in D_{n_k}, \: v \leq w \right\}.
		\end{equation*}	
		Using similar arguments as for $Y_n$, $m^{-n_k}\Delta Y_{n_k} \to \rho W$ $\P^*$ - a.s.		
             We can write for sufficiently large $k$,
		\begin{equation*}
			\P^*[N_{n_k} \leq f_{n_k} \: | \: \mathcal{I}_{k-1}] = \P^*\left[ \left. \max_{v \in \Delta \mathcal{N}_{n_k}} X_v \leq f_{n_k} \right| \mathcal{I}_{k-1} \right] = (1- \P^*[X > f_{n_k} \: | \: \mathcal{T} \:])^{\Delta Y_{n_k}}.
		\end{equation*}
		Since we have
		\begin{equation*}
			m^{n_k}\P^*[X > f_{n_k} \: | \: \mathcal{T} \:] = \left( \gamma  \frac{Y_{n_k}}{m^{n_k}} \right)^{-1} \log n_k (1+o(1)) 
		\end{equation*}
		we can write, for some $\Delta_k \to 0$ a.s.,
		\begin{equation*}
			\P^*[N_{n_k} \leq f_{n_k} \: | \: \mathcal{I}_{k-1}]  = \exp \left( - \gamma^{-1} \log n_k (\delta_k+1) \right) = k^{-\frac{1+\varepsilon}{\gamma}(\Delta_k+1)}.
		\end{equation*}
		By the choice of our parameters, $\frac{1+\varepsilon}{\gamma} <1$. Using a conditional version of the Borel-Cantelli Lemma (see \cite[Theorem 5.3.2]{RD19}) yields that $\P^*$-a.s.
		\begin{equation*}
			N_{n_k} \leq f_{n_k}
		\end{equation*}
		for infinitely many $k$. Letting $k \to \infty$ and $\gamma \downarrow 1$ yields
		\begin{equation*}
			\liminf_{n \to \infty}\frac{N_n - \alpha n^{1/r-1} + \sigma n^{1/r-1}\log\log n }{n^{1/r-1}} \leq \sigma \log\left(a \rho W\right).
		\end{equation*}

	We finally consider the upper space-time envelope. Take $\psi(x)$ such that $\int_0^{\infty} e^{-\psi(x)} dx<\infty$ but $\psi(n) = o(n)$ and consider, for $K \in \R$,
		\begin{equation*}
			g_n  := \alpha n^{1/r} + \sigma n^{1/r-1}(\psi(n)+K).
		\end{equation*}
		Using \eqref{rpower} as always, one can check that $m^n\P[X>g_n] \sim ae^{-K} e^{-\psi(n)}$.
		Take the union bound
		\begin{equation*}
			\P^*[N_n > g_n] \leq m^n \P[X>g_n] = a e^{-K}e^{-\psi(n)}(1+o(1))
		\end{equation*}
		to obtain
		\begin{equation*}
			N_n  \leq \alpha n^{1/r} + \sigma n^{1/r-1}(\psi(n) +K)
		\end{equation*}
		for sufficiently large $n$. If we take $n\to \infty$, followed by $K \to -\infty$, we will obtain the first part of  \eqref{Nnas}.
		Now suppose that $\int_0^{\infty}e^{-\psi(x)} dx =\infty$. Put
		\begin{equation*}
			\mathcal{H}_n = \sigma( X_v, Z_k \: : \: k \in \mathbb{N}, \:|v|\leq n).
		\end{equation*}
		Use the inequality $1-(1-x)^y\geq xy(1-xy)$, $x \in (0,1)$, $y>0$ to obtain 
		\begin{align*}
			\P^*\left[\left.\max_{|v|=n} X_v> g_n \: \right| \: \mathcal{H}_{n-1} \right] & = 1-\P[X \leq g_n]^{Z_n} = 1-( 1-\P[X > g_n])^{Z_n} \\
					& \geq  Z_n\P[X>g_n](1-Z_n\P[X>g_n])  \\ &\geq \const \cdot (W + \delta_n) ae^{-K}e^{-\psi(n)}
		\end{align*} 
		for some $\delta_n \to 0$, $\P^*$-a.s.
		By yet another appeal to the conditional Borel-Cantelli lemma we obtain that infinitely often, a.s.
		\begin{equation*}
			N_n \geq \max_{|v|=n}X_v \geq \alpha n^{1/r} + \sigma n^{1/r-1}( \psi(n)+K).
                      \end{equation*}
                      Again, take $n \to \infty$, followed by $K \to \infty$ to obtain the second part of \eqref{Nnas}.
	\end{proof}

\subsection{Branching random walk}

To show \eqref{mainlarger}, we will prove two inequalities.

\begin{proposition}\label{lbmax}		
		Suppose that Assumptions \ref{assumption_RW}  and \ref{as:BP} are satisfied with $r >\frac 23$. Then,
		\begin{equation*}
			\liminf_{n \to \infty} \frac{M_n - \alpha n^{1/r}}{n^{2-1/r}} \geq \frac{r \log m}{2 \alpha} \qquad \P^*\rm{- a.s.}
		\end{equation*}
\end{proposition}
\begin{proof}
		Let $\varepsilon \in \left(0, \frac 12\right)$ and define
		\begin{equation*}
		 	a_n =\frac{1+2 \varepsilon}{2} \Bigl(\frac{r \log m}{\alpha}\Bigr)^2 n^{3-2/r}, \quad b_n =\frac{(1+2 \varepsilon)^2}{2} \frac{r \log m}{\alpha} n^{2-1/r}
		 \end{equation*}
		 and
		 \begin{equation*}	
		 	c_n = \frac{r \log m}{\alpha} n^{2-1/r}.
		\end{equation*}
		In the first step we show that with probability one, for all $n$ large enough there are many particles in generation $n$ making a large step. 
		Moreover, these particles all have a different ancestor in generation $[\varepsilon n]$. More precisely, for $w \in  D_{[\varepsilon n]}$, let $D^w_{[(1-\varepsilon)n]}$ denote the subset of $D_n$ consisting of descendants of $w$, and define
		\begin{equation*}
			A_n = \{w \in D_{[\varepsilon n]} \colon \exists v_w \in D^w_{[(1-\varepsilon)n]} \text{ s.t.~} X_{v_w} > \alpha n^{1/r} - b_n \}.
		\end{equation*} 
		Note that $v_w$ is a descendant of $w$. We  will show that 
		\begin{equation*}
			\sum_{n=1}^\infty \P^*\left[|A_n|  \leq e^{a_n}\right] <\infty,
		\end{equation*}
which implies that
$\{|A_n|  \leq e^{a_n} \}$ occurs for only finitely many $n$ almost surely. 
	By Lemma~\ref{lemma_W_tail}, there is some $\beta > 0$ such that
		\begin{equation*}
			\sum_{n=1}^\infty \P^*\left[W_n \leq n^{-\beta}\: \mbox{ or } \: W_{[\varepsilon n]} > (\varepsilon n)^\beta\right] <\infty.
		\end{equation*}
		It remains to show that
		\begin{equation}\label{ifWsmall}
			\sum_{n=1}^\infty \P^*\left[|A_n|  \leq e^{a_n}, \: W_{[\varepsilon n]} \leq (\varepsilon n)^\beta, \: W_n > n^{-\beta}\right] <\infty.
		\end{equation}
		For $i \in \N$, let $Z_{[(1- \varepsilon)n]}^{(i)}$ denote the number of descendants in generation $n$ of the $i$th particle from generation $[\varepsilon n ]$. 
		Then 
		\begin{equation*}
			\sum_{i=1}^{Z_{[\varepsilon n]}} Z_{[(1- \varepsilon)n]}^{(i)} = Z_n,
		\end{equation*}
		where the $Z_{[(1- \varepsilon)n]}^{(i)}$'s are independent copies of $Z_{[(1- \varepsilon)n]}$. Using the independence of the branching mechanisms and the displacements, we estimate for sufficiently large $n$, taking into account the inequality $1-x\leq e^{-x}$,
		\begin{align*} 
			&\P^*\left[|A_n|  \leq e^{a_n}, \: W_{[\varepsilon n]} \leq (\varepsilon n)^\beta, \: W_n>n^{-\beta}\right] \\
			&\leq  {\rm const} \cdot \E \left[\1_{\{ W_{[\varepsilon n]} \leq (\varepsilon n)^\beta, \: W_n> n^{-\beta} \}} \binom{Z_{[\varepsilon n]}}{Z_{[\varepsilon n]} - [e^{a_n}]} \prod_{i=1}^{Z_{[\varepsilon n]} - [e^{a_n}]}  
				 \P[X \leq \alpha n^{1/r} - b_n]^{Z_{[(1- \varepsilon)n]}^{(i)}} \right] \\
			&\leq  {\rm const } \cdot \E \left[  \1_{\{ W_n > n^{-\beta}\}} m^{2 \varepsilon n e^{a_n}}   \exp \left\{ -   \P[X > \alpha n^{1/r} - b_n] \sum_{i=1}^{Z_{[\varepsilon n]}-  [e^{a_n}]} Z_{[(1- \varepsilon)n]}^{(i)}  \right\} \right]\\
			&=  {\rm const}\cdot \E \left[  m^{2 \varepsilon n e^{a_n}}   \exp \left\{ -   \P[X > \alpha n^{1/r} - b_n]  \left( Z_n - \sum_{i=1}^{[e^{a_n}]} Z_{[(1- \varepsilon)n]}^{(i)} \right) \right\} \right] \\
			&\leq  {\rm const}\cdot \E \left[  m^{2 \varepsilon n e^{a_n}}   \exp \left\{ -   \P[X > \alpha n^{1/r} - b_n]  \left( m^{n}n^{-\beta} - \sum_{i=1}^{[e^{a_n}]} Z_{[(1- \varepsilon)n]}^{(i)} \right) \right\} \right].
		\end{align*}
		Now note that by Markov's inequality, for $n$ large enough
		\begin{equation*}
			\P^* \left[ \sum_{i=1}^{[e^{a_n}]} Z_{[(1- \varepsilon)n]}^{(i)} > m^{(1-\varepsilon/2)n} \right] \leq\const \cdot m^{-\varepsilon n /2}e^{a_n}\leq m^{-\varepsilon n /3} 
		\end{equation*}
		and therefore we can continue further with 
		\begin{align*}
			&\P^*\left[|A_n|  \leq e^{a_n}, \: W_{[\varepsilon n]} \leq (\varepsilon n)^\beta, \: W_n>n^{-\beta}\right] \\
			& \leq  m^{2 \varepsilon n e^{a_n}}  \exp \left\{ -   \P[X \geq \alpha n^{1/r} - b_n] (m^nn^{-\beta}-  m^{(1-\varepsilon/2)n})    \right\} +m^{-\varepsilon n/3}.
		\end{align*}
		Since
		\begin{equation*}
			\P[X \geq \alpha n^{1/r} - b_n](m^nn^{-\beta}-  m^{(1-\varepsilon/2)n})    \sim a n^{-\beta}e^{\lambda r \alpha^{r-1}n^{1-1/r}b_n} = a n^{-\beta}\exp \left\{ (1+2 \varepsilon)a_n \right\}
		\end{equation*}
we see that indeed \eqref{ifWsmall} holds true because $n^{1+\beta} e^{a_n} = o(e^{(1+2\varepsilon)a_n})$.\\
	
In the second step we consider $\max_{w \in A_n}\left(S_{v_w} - X_{v_w} - S_w\right)$. Note that the random walks $( S_{v_w} - X_{v_w} - S_w)_{w \in A_n}$ are independent and independent of $\{X_{v_w}, w \in A_n\}$ and have the same distribution as $S_{[(1-\varepsilon)n]}$. We show that $\max_{w \in A_n} \left(S_{v_w} - X_{v_w} - S_w\right) \leq c_n$ 
		occurs only finitely often almost surely. Write
		\begin{align*}
			\P^* \left[\max_{w \in A_n} \left(S_{v_w} - X_{v_w} - S_w\right) \leq c_n\right] \leq & \ \P^*\left[\left.\max_{w \in  A_n} \left(S_{v_w} - X_{v_w} - S_w \right)\leq c_n \right| \ |A_n| > e^{a_n}\right]\\
				& \ + \P^*[|A_n| \leq e^{a_n}].
		\end{align*}
		But 
		\begin{align*}
			\P^*\left[\left.\max_{w \in A_n} \left(S_{v_w} - X_{v_w} - S_w\right) \leq c_n \right| |A_n| > e^{a_n}\right] & \leq \Bigl(1- \P[S_{[(1- \varepsilon)n]} > c_n] \Bigr)^{e^{a_n}}\\
				& \leq \exp \left\{ -e^{a_n} \P[S_{[(1- \varepsilon)n]} > c_n] \right\}\\
				&=\exp \left\{ -\exp \left(a_n - \frac{c_n^2}{2(1- \varepsilon)n } (1+o(1) \right) \right\}
		\end{align*} 
		where we used Lemma~\ref{lemma_RW} for the last equality.
		The last expression is summable by the choice of $\varepsilon$. 
		In the third step we prove that 
		\begin{equation*}
			\liminf_{n \to \infty} n^{2-1/r}\min_{w \in A_n} S_w \geq -\varepsilon^{1/3}
		\end{equation*}
		provided $\varepsilon$ is small enough.
		For sufficiently large $n$ we can write
		\begin{align*}
			\P^*\left[\min_{w \in A_n} S_w < -\varepsilon^{1/3} n^{2-1/r} \right] & \leq {\rm const } \: m^n \P[X> \alpha n^{1/r} -b_n] \P[S_{[\varepsilon n]} < -\varepsilon^{1/3} n^{2-1/r}] \\
				& \leq {\rm const} \exp \left\{ \frac{2 r \log m }{\alpha \sigma} n^{3-2/r} - \frac{\varepsilon^{-1/3}}{3} n^{3-2/r} \right\},
		\end{align*}
		where in the last inequality we applied Lemma~\ref{lemma_RW} yet again. Taking $\varepsilon$ sufficiently small secures our claim. 
		All three steps together imply that
		\begin{equation*}
			\liminf_{n \to \infty} \frac{M_n - \alpha n^{1/r}}{n^{2-1/r}} \geq \lim_{n \to \infty}\frac{c_n-b_n}{n^{2-1/r}} - \varepsilon^{1/3} =  \left( 1 - \frac{(1+2\varepsilon)^2}{2} \right) \frac{r \log m}{2\alpha} - \varepsilon^{1/3}.
		\end{equation*}
		Letting $\varepsilon \to 0$ finishes the proof.
	\end{proof}

\begin{proposition}\label{asindupper}
		Suppose that Assumptions \ref{assumption_RW}  and \ref{as:BP} are satisfied and $r > \frac 23$. Then
		 \begin{equation*}
 			\limsup_{n \to \infty} \frac{M_n - \alpha n^{1/r}}{n^{2-1/r}} \leq \frac{r \log m}{2 \alpha} \quad \P^*-a.s. 
 		\end{equation*}
	\end{proposition}
	
	\begin{proof}
		Let $\varepsilon>0$. Using a union bound
\begin{equation*}
\P^*\Bigl[M_n - \alpha n^{1/r} \geq \frac{(1+\varepsilon) r \log m}{2 \alpha}  n^{2-1/r} \Bigr] 
\leq \P\Bigl[S_n - \alpha n^{1/r} \geq \frac{(1+\varepsilon) r \log m}{2 \alpha} n^{2-1/r} \Bigr] \frac{m^n}{1-q}.
\end{equation*}
		It remains to show that the r.h.s. is summable, and this is the statement of Lemma \ref{sumstat}.
		
\end{proof}

	We can now turn to the case $ r \leq \frac 23$, and
        prove Lemma \ref{comparelemma}.
        To analyse $M_n$ we need to partition $D_n$ into four classes of particles. The first one consists of those particles with no big displacements along their ancestral line, i.e.
	\begin{equation*}
		\mathcal{A}_n = \left\{ w \in D_n \: : \: \forall v \in [o, w], \: X_v \leq \delta n^{1/r} \right\},  
	\end{equation*} 
	where $\delta \in (\alpha 2^{-1/r},\alpha)$ is fixed. The next class consists of those particles that had (at least) two big displacements along their ancestral line, i.e.
	\begin{equation*}
		\mathcal{B}_n = \left\{ w \in D_n \: : \: \exists v, u \in [o, w], \: v \neq u, \text{ s.t. }  \: \min \{ X_v, X_u\} > \delta n^{1/r} \right\}.
	\end{equation*} 		
	All other particles have exactly one big displacement along their ancestral line. We will need to distinguish further if this
 jump is greater or smaller than
	\begin{equation}\label{s_ndefi}
		s_n = \alpha n^{1/r} -  T n^{1/r-1}\log n
	\end{equation}
	where $T$ is fixed to be sufficiently large, that is $T  > \frac{1+3 \alpha^{1-r} \lambda^{-1}}{(1-r)}$. Define
	\begin{equation*}
		\mathcal{C}_n = \left\{ w \in D_n \: : \:  \exists v \in [o,w] \text{ s.t. } X_v \in (\delta n^{1/r}, s_n], \text{ and } \forall u \in [o, w]\setminus \{v\}, \, X_u \leq \delta n^{1/r} \right\}
	\end{equation*}		
	and
	\begin{equation*}
		\mathcal{D}_n= \left\{ w \in D_n \: : \:  \exists v \in [o,w] \text{ such that } X_v > s_n, \text{ and } \forall u \in [o, w]\setminus \{v\} \, X_u \leq \delta n^{1/r} \right\}. 
	\end{equation*}		
	By the merit of Proposition~\ref{prop:Nn_LIL}, for sufficiently large $n$,  $\mathcal{D}_n$ is non-empty $\P^*$ - a.s. If we denote
	\begin{equation*}
		M_n^{\mathcal{A}} = \max_{w \in \mathcal{A}_n }S_w, \quad M_n^{\mathcal{B}} = \max_{w \in \mathcal{B}_n }S_w, \quad  M_n^{\mathcal{C}} = \max_{w \in \mathcal{C}_n }S_w, \quad M_n^{\mathcal{D}} = \max_{w \in \mathcal{D}_n }S_w
	\end{equation*}
	we can write
	\begin{equation*}
		M_n = \max \left\{ M_n^{\mathcal{A}}, M_n^{\mathcal{B}}, M_n^{\mathcal{C}}, M_n^{\mathcal{D}} \right\}.
	\end{equation*}
We will see that the relevant term is $M_n^{\mathcal{D}}$.
	\begin{lemma}\label{Mnaest}
		Let Assumptions \ref{assumption_RW}  and \ref{as:BP} hold for some $r \in \left(0,\frac{2}{3}\right]$. We have $\P^*$-a.s. for sufficiently large $n$,
		\begin{equation*}
			M_n^{\mathcal{A}} \leq \alpha n^{1/r} - n^{1/r-1}\log n. 
		\end{equation*}
	\end{lemma}

	\begin{proof}
		Let
		\begin{equation}\label{hatdef}
			\hat{X}_k = X_k \1_{\left\{  X_k< \delta n^{1/r} \right\}}, \quad \hat{S}_n = \sum_{k=1}^n\hat{X}_k, \quad \gamma_n =  \alpha n^{1/r} - n^{1/r-1}\log n.
		\end{equation}
		Using the Markov inequality and Lemma~\ref{lem:12} we can estimate $\P^*\left[M_n^{\mathcal{A}}> \gamma_n\right]$ in the following way:
		\begin{align*}
                  \P^*\left[M_n^{\mathcal{A}}>\gamma_n\right]& \leq \const \cdot m^n\P\left[\hat{S}_n > \gamma_n \right]\\
                  & = \const \cdot m^n\P\left[ \lambda \frac{\alpha^r n + n^{1/r-1}}{\gamma_n} \hat{S}_n > \lambda(\alpha^r n +n^{1/r-1}) \right] \\
				& \leq  \const\cdot \exp \left\{ - \lambda n^{1/r-1} \right\} \E\left[\exp\left\{ \lambda \frac{\alpha^r n +n^{1/r-1}}{\gamma_n } \hat{S}_n \right\} \right]\\
				& = \const \cdot \exp \left\{ - \lambda n^{1/r-1} \right\} \E\left[\exp\left\{ \lambda \frac{\alpha^r n + n^{1/r-1}}{ \gamma_n} \hat{X}_1 \right\} \right]^n\\
				& \leq \const \cdot \exp\left\{ - \lambda n^{1/r-1} \right\} \exp\left\{ n\lambda^2  \frac{(\alpha^r n + n^{1/r-1})^2}{2\alpha^2n^{2/r}}  + o(n^{3-2/r})  \right\} \\
				& = \const \cdot \exp\left\{ -\lambda n^{1/r-1} (1 + o(1) ) \right\}.
		\end{align*} 
		This shows that
		\begin{equation*}
			\sum_{n=1}^{\infty} \P^*\left[M_n^{\mathcal{A}}>\gamma_n\right] <\infty
		\end{equation*}
		and concludes the proof.
	\end{proof}

	\begin{lemma}
		Let Assumptions \ref{assumption_RW}  and \ref{as:BP} hold for some $r \in (0,1)$. Then, for sufficiently large $n$, $\P^*$-a.s.
		\begin{equation*}
			 \mathcal{B}_n=\emptyset\, .
		\end{equation*}
	\end{lemma}
	\begin{proof}
		We have 
		\begin{equation*}
			\P^*\left[ |\mathcal{B}_n| \geq 1 \right] \leq \E^*\left[ |\mathcal{B}_n| \right] \leq n^2 m^n\P\left[X > \delta n^{1/r} \right]^2 \leq  \const \cdot n^2 \exp \left\{ \lambda( \alpha^r - 2 \delta^r)n \right\}
		\end{equation*}
		where the exponent in the last term is negative by the choice of $\delta$. 
	\end{proof}

	\begin{lemma}
		Let Assumptions \ref{assumption_RW} and \ref{as:BP} hold for $r \in (0,1)$.
		We have $\P^*$-a.s.
		\begin{equation*}
			M_n^{\mathcal{C} }\leq  \alpha n^{1/r}-  n^{1/r-1} \log n
		\end{equation*}
		for sufficiently large $n$.
	\end{lemma}
	\begin{proof}
		 To see that this is true recall $\gamma_n$ from \eqref{hatdef} and $s_n$ form~\eqref{s_ndefi}, set
		 $\hat{X}_k = X_k \1_{\left\{  X_k< \delta n^{1/r} \right\}}$ 
		 and put  
		\begin{equation*}
			\tilde{X}_n = X_n \1_{\left\{ \delta n^{1/r}< X_n \leq s_n \right\}}, \quad 
			\tilde{S}_{n-1} = \sum_{k=1}^{n-1}\hat{X}_k, 
		\end{equation*}
		and set $H_n = \left\{ M_n^{\mathcal{C} }> \gamma_n \right\}$. We have
		\begin{multline*}
			\P^*[H_n] \leq \const\cdot n m^n \P \left[ \tilde{S}_{n-1} + X_n > \gamma_n, \delta n^{1/r} <X_n \leq  s_n \right] \\
			 =  \const\cdot n m^n  \P \left[ \lambda \frac{\alpha^r n +\frac{3}{\lambda} \log n }{\gamma_n} \left(  \tilde{S}_{n-1} + X_n \right) > \lambda \left( \alpha^r n +\frac{3}{\lambda} \log n \right),  \delta n^{1/r} <X_n \leq  s_n  \right].
		\end{multline*}
		Apply the Markov inequality and Lemma \ref{lem:12} for a bound for the exponential moment of $\tilde{S}_{n-1}$ (as we did it for $\hat{S}_n$ in the proof of Lemma \ref{Mnaest}) to obtain
		\begin{equation*}
			\P^*[H_n]  \leq {\rm const} \cdot   n^{-2}\E\left[\exp \left\{  \lambda \frac{\alpha^r n + \frac 3\lambda\log n}{\gamma_n} X_n \right\}   \1_{\left\{ \delta n^{1/r}< X_n \leq s_n \right\}} \right].
		\end{equation*}
		It remains to show that
		\begin{equation}\label{eq:forsprim}
			\E\left[\exp \left\{ \lambda \frac{\alpha^r n + \frac 3\lambda\log n}{\gamma_n} X_n \right\} \1_{\left\{ \delta n^{1/r}< X_n \leq s_n \right\}} \right]
		\end{equation}
		is bounded. 
		Use the inequality
		\begin{equation*}
			\E [\psi(X)\1_{\{ \delta n^{1/r} < X \leq s_n\}} ] \leq \int_{\delta n^{1/r}}^{s_n} \psi'(s) \P[X>s]ds + \psi(\delta n^{1/r})\P[X>\delta n^{1/r}]
		\end{equation*}
		with $\psi(s) = \exp\left\{  \lambda \frac{\alpha^r n +3\lambda^{-1}\log n}{\gamma_n} s\right\} $. Since 
		\begin{equation*}
			\psi(\delta n^{1/r})\P[X>\delta n^{1/r}]\leq  \exp \left\{ \lambda \delta \left(\alpha^{r-1} -\delta^{r-1}\right)n + o(n) \right\}
		\end{equation*} 
		we will focus on the integral for which we have
		\begin{align*}
			&\int_{\delta n^{1/r}}^{s_n} \psi'(s) \P[X_n>s]ds  \leq {\rm const}\cdot  n^{1-\frac 1r} \int_{\delta n^{1/r}}^{s_n}  \exp \left\{ \lambda \frac{\alpha^r n + 3 \lambda^{-1}\log n }{\gamma_n} s - \lambda s^r \right\} ds \\
				& \leq{\rm const}\cdot  n  \int_{\delta\alpha^{-1} + o(1)}^1\exp \left\{  \lambda \frac{s_n}{\gamma_n} s( \alpha^r n +  3\lambda^{-1} \log n  -  \gamma_n s_n^{r-1}) \right\} ds.
		\end{align*}
		To check that the last term is bounded consider the exponent 
		\begin{align*}
			 \alpha^r n + \frac 3\lambda \log n  -  \gamma_n s_n^{r-1} & = \frac 3\lambda \log n  - \left( (1-r)T-1\right)\alpha^{r-1}\log n   + o(1).
		\end{align*}	
		We see that whenever 
		\begin{equation*}
			T > \frac{1+3 \alpha^{1-r} \lambda^{-1}}{(1-r)}
		\end{equation*}
		the expression in the integral is bounded by
		\begin{equation*}
			\exp \left\{  \lambda \frac{s_n}{\gamma_n} s\left( \alpha^r n + \frac 3\lambda \log n  -  \gamma_n s_n^{r-1}\right) \right\}  \leq \exp \left\{- (2 \alpha \delta^{-1}+o(1))s \log n \right\} \leq n^{-2+o(1)},
		\end{equation*}
		where the last inequality is a consequence of $s>\delta\alpha^{-1}+o(1)$.
		Thus, the integral compensates the factor $n$, so~\eqref{eq:forsprim} is indeed bounded in $n$.
	\end{proof}

	\begin{lemma}\label{ref:lem23}
		Under Assumptions \ref{assumption_RW}  and \ref{as:BP} with $r <\frac 23$, 
		\begin{equation*}
			\frac{M^{\mathcal{D}}_n - N_n}{n^{1/r-1}} \to 0, \qquad \P^*-\rm{ a.s.}
		\end{equation*}
	\end{lemma}  
	\begin{proof}
		Recall $s_n$ defined in~\eqref{s_ndefi}. Fix $\varepsilon>0$ and first estimate the probability that the difference is large. Note that $X_v-N_n \leq 0$ for all $v \in \mathcal{N}_n$ and thus, with $S_w^{\setminus v} = \sum_{u \in [o,w]\setminus \{v\}} X_u$,
using Lemma~\ref{lemma_RW}
\begin{align*}
			&\P^*\left[  M_n^{\mathcal{D}}-N_n  >\varepsilon n^{1/r-1}\right] \leq  \\  \nonumber
			&\P^* \left[ \exists w \in D_n, \: \exists v \in [o,w] \text{ s.t. } X_v> s_n, \, \forall u \in [o,w]\setminus \{v\}, X_u\leq \delta n^{1/r}, \text{ and }S_w^{\setminus v}   > \varepsilon n^{1/r-1}  \right] \\ \nonumber
			& \leq n m^n \P[X>s_n] \P\left[ \tilde S_{n-1}>\varepsilon n^{1/r-1}\right] \leq n m^n \P[X>s_n] \P\left[ \tilde S_{n-1}>K \sqrt{n\log n}\right] \\ \nonumber
			& = n \cdot a n^{T/\sigma} n^{-K^2/2(1+o(1))} \to 0 
		\end{align*}
		with some constant $K$ which is sufficiently large. 	
On the other hand if the difference $M^{\mathcal{D}}_n - N_n$ is small, this means that for each $w^* \in \mathcal{D}_n$ and $v^* \in \left[o,w^*\right]$ such that $X_{v^*}=N_n$, it must hold that
		\begin{equation*}
			\sum_{u \in [o,w^*]\setminus \{ v^*\}} X_u \leq -\varepsilon n^{1/r-1}.
		\end{equation*}   
		Since, by an appeal to Proposition~\ref{prop:Nn_LIL} there always exists at least one such $w^*$, we have
			\begin{align*}
			\P^*\left[  M_n^{\mathcal{D}}-N_n  <- \varepsilon n^{1/r-1}\right] & \leq  n m^n \P[X>s_n] \cdot  \P\left[ \tilde S_{n-1} <- \varepsilon n^{1/r-1}\right] \\
				& \leq \const\cdot n^{1+ T/\sigma -K^2/2(1+o(1))}.
		\end{align*}	
 \end{proof}    
	Putting together Proposition \ref{prop:Nn_LIL} and Lemmas \ref{Mnaest} -\ref{ref:lem23} we get Lemma~\ref{comparelemma}.	
	To treat the case $r = \frac 23$ we will need a finer decomposition of $M^{\mathcal{D}}_n$. 
	
	\begin{proposition}
		Let Assumptions \ref{assumption_RW}  and \ref{as:BP} be in force. If $r =\frac 23$ then
		\begin{equation}
			\frac{M_n - \alpha n^{3/2}}{\sigma\sqrt{n}} \overset{d}\to  V_{2/3}
            \end{equation}
where the c.d.f. of $V_{2/3}$ is given by \eqref{mainsmallrcrit}.
		
	\end{proposition}
	
	\begin{proof}
		Recall \eqref{s_ndefi}, take $C_1 > T/(\sigma \log m)$,  and consider the event 
		\begin{equation*}
			A_n = \{ N_{[n-C_1 \log n] } > s_n\}.
		\end{equation*}
		As one computes directly,
		\begin{equation*}
			\P^* [ A_n ] \leq \const \cdot n^{-C_1 \log m +\sigma^{-1} T} \to 0.
		\end{equation*}
		In words, with high probability, whenever $w \in \mathcal{D}_n$ the ancestor $v$ of $w$ for which $X_v>s_n$ must come from generation at least $[n - C_1 \log  n] $. 
		Recall that for $x,y \in \mathcal{T}$ we denote by $x\wedge y$ the last common ancestor of $x$ and $y$. 
		Take $C_2 > 2 T/(\sigma \log m)$ and consider the event
		\begin{equation*}
			B_n= \{ \exists v, w \in \mathcal{T},\text{ such that } v\neq w, \: |v|, |w|\leq n, \: |v\wedge w| \geq C_2 \log n , \: X_v\wedge X_w >s_n\}.
		\end{equation*} 
		Then, since we can choose $v$ in roughly $m^n$ ways and then choose $w$ in roughly $m^{n-C_2 \log n}$ ways, we have
		\begin{equation*}
			\P^*[B_n] \leq \const \cdot m^{2n-C_2 \log n} \P[X>s_n]^2 \leq \const \cdot n^{-C_2\log m + 2\sigma^{-1}T} \to 0.
		\end{equation*}
		This means that with high probability any two particles with big displacements must be distantly related, i.e. the graph distance in $\mathcal{T}$ between the vertices in question must be sufficiently large. Let
		\begin{equation*}
			\hat{M}^{\mathcal{D}}_n=\max_{w \in \mathcal{D}_n} \left\{ \left| \sum_{ \substack{ {u \in [o,w],} \\ {  |u| \notin [C_2\log n, n-C_1 \log n ]}}} X_u \1_{\{X_u \leq \delta n^{1/r} \}} \right|\right\}.
		\end{equation*}
		Then we claim that
		\begin{equation*}
			\frac{\hat{M}^{\mathcal{D}}_n}{\sqrt{n}} \to 0 \qquad \P^*\text{- a.s.}
		\end{equation*}
		Indeed, using a union bound we can write 
		\begin{align*}
			\P^* \left[ \hat{M}_n^{\mathcal{D}} > \varepsilon \sqrt{n} \right] & \leq m^n\P[ X>s_n] \P\left[|S_{[(C_1+C_2)\log(n)]}| > \varepsilon \sqrt{n}\right] \\
				& \leq {\rm const}\cdot  \log n\cdot m^n \P[ X>s_n] \P\left[X> \varepsilon \sqrt{n} (\log n)^{-1}\right] \\
				&\leq {\rm const}\cdot  \log n\cdot n^{\rm const} \cdot \E[ |X|^{j_0}] n^{-j_0/2} (\log n)^{j_0}
		\end{align*}
		and, using \eqref{momentass}, the last expression is summable provided that $j_0$ is large enough. 
		Finally, consider
		\begin{equation*}
			\tilde{M}^{\mathcal{D}}_n=\max\left\{  X_v + \!\!\!\!\!\!\!\!\!\! \sum_{ \substack{ {u \in [o,w],} \\ {  |u| \in [C_2\log n, n-C_1 \log n ]}}}\!\!\!\!\!\!\!\!\! X_u\: : \: w \in \mathcal{D}_n \text{ s.t. } \exists v \in [o,w], \: |v| > n-C_1 \log n, X_v>s_n\right\}.
		\end{equation*}
		Since $|\tilde{M}^{\mathcal{D}}_n-M^{\mathcal{D}}_n| \leq \hat{M}^{\mathcal{D}}_n$ the above considerations imply that
		\begin{equation*}
			\frac{\tilde{M}^{\mathcal{D}}_n-M^{\mathcal{D}}_n}{ \sqrt{n}} \overset{\P^*}\to 0
		\end{equation*}
		and therefore it is sufficient to prove weak convergence of $\tilde{M}^{\mathcal{D}}_n$. Put
		\begin{equation*}
			\tilde{\Phi}_n(s) = \P\left[   S_{[n-(C_1+C_2)\log n]} \leq s\sigma\sqrt{n}, \: X_i \leq \delta n^{1/r} \mbox{ for } 1\leq i \leq n\right].
		\end{equation*}
		Note that the $X_v$'s that appear in the definition of $ \tilde{M}^{\mathcal{D}}_n$ must be some of the extremes in the collection $\{ X_v \}_{v \in \mathcal{N}_n}$ and therefore
		\begin{align*}
			& \P^* \left[ \left.  B_n^c \cap \left\{ \tilde{M}^{\mathcal{D}}_n \leq \alpha n^{3/2} + x\sigma \sqrt{n} \right\}  \right| \mathcal{T} \right]\\
			&= \E^* \left[ \left.\P^* \left[ \left.  B_n^c \cap \left\{ \tilde{M}^{\mathcal{D}}_n \leq \alpha n^{3/2} + x\sigma \sqrt{n} \right\}  \right| \: Z, \: X_v, \: |v|\geq n-C_1\log n \right] \right|  \mathcal{T} \right] \\
			& = \E^* \left[ \left. \prod_{v\in \mathcal{N}_n : X_v>s_n} \tilde{\Phi}_n \left( x - \frac{X_v -\alpha n^{3/2}}{\sigma \sqrt{n}} \right)  \cdot \1_{B_n^c} \right| \mathcal{T} \right] \\
			&=  \E^* \left[ \left.  \1_{B_n^c} \exp \left\{ \int\limits_{-T\sigma^{-1}\log n }^\infty \log\left(\tilde{\Phi}_n(x-y)\right) \: \Lambda_n(dy) \right\}  \right|  \mathcal{T}  \right] .
		\end{align*}
		Since $ \1_{B_n^c} \overset{\P^*}\to 1$ it is enough to argue that conditioned on $\mathcal{T}$ (recalling that
		$\Phi$ denotes the c.d.f. of a centred Gaussian distribution with variance $\sigma^{-2}$),
		\begin{equation*}
			 \int\limits_{-T\sigma^{-1}\log n }^\infty \log\left(\tilde{\Phi}_n(x-y)\right) \: \Lambda_n(dy)  \overset{d}\to  \int\limits_{-\infty}^\infty \log(\Phi(x-y)) \: \Lambda(dy).
		\end{equation*}
		If we put
		 \begin{equation*}
			\Phi_n(s) = \P\left[ \left.  S_{[n-(C_1+C_2)\log n ]} \leq s\sigma \sqrt{n} \right|  X_i \leq \delta n^{1/r}\mbox{ for } 1 \leq i \leq n\right]
		\end{equation*}
		then $\Phi_n(s) \to \Phi(s)$, and
		\begin{align*}
		 	&\int\limits_{-T\sigma^{-1}\log n }^\infty \log\left(\tilde{\Phi}_n(x-y)\right) \: \Lambda_n(dy) \\  
		 			=& \int\limits_{-T\sigma^{-1}\log n}^\infty \log\left(\Phi_n(x-y)\right) \: \Lambda_n(dy) \\
		 						&+\log \left( \P^*\left[X \leq \delta n^{1/r}\right]^{n-(C_1+C_2)\log n} \right) \Lambda_n (-T\sigma^{-1}\log n, \infty).
		\end{align*}
		The last term vanishes since by the  Markov inequality and the fact that conditioned on $\mathcal{T}$, $\Lambda_n (-T\sigma^{-1}\log n, \infty)$ is a binomial random variable
		\begin{align*}
			\P^* \left[ \Lambda_n (-T\sigma^{-1}\log n, \infty) > n^{2 + T\sigma^{-1}}\right] & \leq   n^{-2 - T\sigma^{-1}} \E^* \left[ \Lambda_n (-T\sigma^{-1}\log n, \infty) \right] \\
				& \leq \const \cdot n^{-2 - T\sigma^{-1}} m^n \P\left[X > \alpha n^{1/r} - T n^{1/r-1} \log n \right]  \\ &\leq \const \cdot n^{-2}
		\end{align*}
		and thus $ n^{-3-T\sigma^{-1}} \Lambda_n (-T\sigma^{-1}\log n, \infty) \to 0$ $\P^*$-a.s. which implies that $\P^*$-a.s.
		\begin{equation*}
			\log\left( \P^*\left[X \leq \delta n^{1/r}\right]^{n-(C_1+C_2)\log n} \right) \Lambda_n (-T\sigma^{-1}\log n, \infty) \to 0.
		\end{equation*}
		In order to analyse $  \int_{-T\sigma^{-1}\log n}^{\infty} \log\left(\Phi_n(x-y)\right) \: \Lambda_n(dy)$ we will first introduce a point process $\Lambda_n^*$ which is a marked version of $\Lambda_n$, show that it is convergent and then explain how $\Lambda^*_n$ is related to our random integral.
		Consider a family of iid random variables $\left\{ U_v^{(n)} \right\} _{ v \in \mathcal{T} }$ independent from $Z$ and $\{X_v\}_{v \in \mathcal{T}}$ with common distribution $\Phi_n$ and define a process on $\R^2$ given via
		\begin{equation*}
			\Lambda_n^* = \sum_{v \in \mathcal{N}_n} \epsilon_{\left(\bar{X}_v ,U_v \right)}, \text{ where } \bar{X}_v = \frac{X_v -\alpha n^{1/r}}{\sigma \sqrt{n}}.
		\end{equation*}
		Then, conditioned on $Z$, $\Lambda_n^*$ is a binomial point process. 
		Note that, since $r =\frac{2}{3}$, for $A = (t_1,t_2] \times (s_1, s_2]$,
		\begin{equation*}
			  m^n\P\left[\left( \frac{X -\alpha n^{1/r} }{\sigma \sqrt{n}}, U^{(n)}\right) \in A \right] \to \int_A \frac{a\sigma}{\sqrt{2 \pi }}e^{-t^2 \sigma^2/2}e^{-s} \: dtds.
		\end{equation*}
		Using exactly the same arguments as in the proof of Proposition \ref{prop:Nn} for the convergence $\Lambda_n \to \Lambda$, one can show that conditioned on $\mathcal{T}$, $\Lambda_n^* \to \Lambda^*$, where conditioned on $\mathcal{T}$, $\Lambda^*$ is a Poisson random measure with intensity 
		\begin{equation*}
			\mu^*(W, dt,ds) = a\rho W  \frac{\sigma}{\sqrt{2 \pi }}e^{-s^2 \sigma^2/2}e^{-t} \: dtds. 
		\end{equation*} 
		Now note that
		\begin{equation*}
			 \E^* \left[ \left.\exp \left\{  \int_{-T\sigma^{-1}\log n} \log\left(\Phi_n(x-y)\right) \: \Lambda_n(dy) \right\} \right|  \mathcal{T}  \right] = \P^* \left[ \left. \Lambda^*_n(A_{n,x}) =0 \right|  \mathcal{T} \right],
		\end{equation*}
		where
		\begin{equation*}
			A_{n,x} = \left\{ \left. (t,s) \in \R^2 \: \right| \: t+s > x, \: t \geq - T \sigma^{-1}\log n  \right\}. 
		\end{equation*}
		We will argue that $\P^*$-as
		\begin{equation}\label{jointconv}
			 \P^* \left[ \Lambda^*_n(A_{n,x}) =0 | \mathcal{T} \right] \to  \P^* \left[\Lambda^*(A_{\infty,x}) =0 | \mathcal{T}\right],
		\end{equation}
		where $A_{\infty,x} = \bigcup_{n \geq 1} A_{n,x}$. Let for $R>0$, $B_R = [-R,R]^2 \subseteq \R^2$. By the merit of the weak convergence of $\Lambda_n^*$ to $\Lambda^*$,
		\begin{equation}\label{eq:4:claim}
			\P^* \left[  \Lambda^*_n(A_{n,x}\cap B_R) =0  | \mathcal{T} \right]  \\  \to \P^* \left[\Lambda^*(A_{\infty,x}\cap B_R) =0 | \mathcal{T} \right]\quad \P^*-\mbox{as}.
		\end{equation}
		Next, note that for some $\Delta_n \to 0$ a.s. and some sufficiently large constant ``${\rm const}$'', we have almost surely
		\begin{align*}
			\P^*\left[ \left. \Lambda_n^* \left(  A_{n,x} \cap( ( R, \infty)\times \R) \right) >0 \: \right| \: \mathcal{T} \: \right] & \leq Y_n \P\left[ X>\alpha n^{1/r} + R\sigma\sqrt{n} \right] \\
				& \leq \const \cdot (1+\Delta_n)W e^{-R}.
		\end{align*}
		To treat the other component of $(A_{n,x}\cap B_R)^c$ write
		\begin{align*}
			&\P^*\left[ \left. \Lambda_n^* \left(  A_{n,x} \cap(  \R \times ( R, \infty) \right) >0 \: \right|  \mathcal{T}  \right] \\
				& \leq Y_n \P\left[ X>s_n, \: U^{(n)}> R , \: X +\sigma\sqrt{n} U^{(n)} > \alpha n^{1/r} + x\sigma\sqrt{n} \right] \\
				& \leq \const \cdot (1+\Delta_n)W  m^n  \P\left[ X>s_n, \: U^{(n)}> R , \: X +\sigma\sqrt{n} U^{(n)} > \alpha n^{1/r} + x\sigma\sqrt{n} \right].
		\end{align*}
		We can estimate the last term via
		\begin{align*}
			& m^n  \P\left[ X>s_n, \: U^{(n)}> R, \: X +\sigma\sqrt{n} U^{(n)} > \alpha n^{1/r} + x\sigma\sqrt{n} \right] \\ 
			& \leq m^n  \P\left[   \frac{X-\alpha n^{1/r}}{\sqrt{n} } \geq 0, \: U^{(n)}> R  \right] \\
			&+\sum_{j= 1}^{T \log n} m^n  \P\left[   \frac{X-\alpha n^{1/r}}{\sqrt{n} } \in (- j, - (j-1)), \: U^{(n)}> R, \: X + \sigma\sqrt{n} U^{(n)} \geq \alpha n^{1/r} + \sigma x \sqrt{n} \right]\\
			&\leq  \const \cdot e^{-\sigma^2R^2/2} +  \const \cdot \frac{(\log n)^2}{\sqrt{n}} + \const  \cdot \sum_{j= 1}^{T\log n } e^{\sigma j} \cdot  e^{ -((\sigma x +j)\wedge R \sigma)^2/2}\\
			&\leq \const \cdot e^{-R/2}, 
		\end{align*}
		where the last inequality holds provided that $R>0$ is sufficiently big. Therefore  
		\begin{equation*}
			 \P^* \left[ \left. \Lambda^*_n(A_{n,x}) =0 \right|  \mathcal{T} \right] = \P^*\left[ \left. \Lambda_n^* \left(  A_{n,x} \cap B_R \right) =0 \: \right|  \mathcal{T}  \right] + O(e^{-R}). 
		\end{equation*}
		Taking $n \to \infty$ followed by $R \to \infty$ proves~\eqref{jointconv}. This concludes the proof since $\Lambda^*$ conditioned on $\mathcal{T}$ is a Poisson random measure and so
		\begin{align*}
			  \P^* \left[ \left. \Lambda^*(A_{\infty,x}) =0 \right| \: Z \: \right]  &= \exp \left\{  -  \mu^* (W, A_{\infty,x}) \right\}  \\
			 &=  \exp \left\{ - a\rho W  \int_{A_{\infty,x}} \frac{\sigma}{\sqrt{2 \pi }}e^{-s^2 \sigma/2}e^{-t} \: dsdt \right\} \\
			 &= \exp \left\{ - a\rho W  \int (1-\Phi(x-t))e^{-t} \: dt \right\}.
		\end{align*}
	\end{proof}

	We finally consider the lower and upper space-time envelopes. 
\begin{proof}[Proof of \eqref{limsupandinf}]
		We first establish that	$\P^*$-a.s.
		\begin{equation}\label{eq:4:lastclaim}
			-\infty <  \liminf_{n \to \infty } \frac{M_n^{\mathcal{D}} -N_n}{ \sqrt{n\log n}} \leq \limsup_{n \to \infty } \frac{M_n^{\mathcal{D}} -N_n}{ \sqrt{n\log n}} <  \infty.
		\end{equation}
		This can be shown using the same arguments as in the proof of Lemma~\ref{ref:lem23}. Indeed, we can use the first formula in the proof of Lemma~\ref{ref:lem23}, to get for $K>0$,
		\begin{align*}
			\P^*\left[  M_n^{\mathcal{D}}-N_n  >K \sqrt{n\log n}\right] & \leq   n m^n \P[X>s_n] \P\left[ \tilde S_{n-1}>K \sqrt{n\log n}\right] \\
			& = n \cdot a n^{T/\sigma} n^{-K^2/2(1+o(1))} \to 0
		\end{align*}
		provided that $K$ is taken sufficiently large. Similarly, as in the last display of the proof of Lemma~\ref{ref:lem23},
			\begin{align*}
			\P^*\left[  M_n^{\mathcal{D}}-N_n  <- K \sqrt{n\log n} \right] & \leq  n m^n \P[X>s_n] \cdot  \P\left[ \tilde S_{n-1} <-K \sqrt{n\log n}\right] \\
				& = n^{1+ T/\sigma -K^2/2(1+o(1))}.
		\end{align*}	
		The first formula in \eqref{limsupandinf} follows if we combine~\eqref{eq:4:lastclaim} with 
		\begin{equation*}
			\limsup_{n \to \infty} \frac{N_n-\alpha n^{3/2}}{ \sqrt{n} \log n } = \sigma
		\end{equation*}
		which comes from the last part of Proposition~\ref{prop:Nn_LIL} by testing with $\psi(x) = (1 \pm \varepsilon)\log x $. The 
		second formula in \eqref{limsupandinf} lower follows from~\eqref{eq:4:lastclaim} and
		\begin{equation*}
					\liminf_{n \to \infty} \frac{N_n-\alpha n^{3/2}}{ \sqrt{n \log n} } = 0
		\end{equation*}
		which comes from the first part of~Proposition~\ref{prop:Nn_LIL}.
	\end{proof}

\section*{Acknowledgement}
Piotr Dyszewski  was partially supported by the National Science Centre, Poland
 (Sonata Bis, grant number DEC-2014/14/E/ST1/00588). This work was initiated while the first author was visiting the Department of Mathematics, Technical University of Munich in February 2019. He gratefully acknowledges financial support and hospitality.
 We are very much indepted to two anonymous referees for reading carefully and detecting several glitches in the first version of the article.

\section*{Appendix}	
In this appendix, we provide the proof of Lemma \ref{lem:12} and Lemma \ref{sumstat}.
\begin{proof}[Proof of Lemma  \ref{lem:12}] The arguments are similar as in the proof of (18) in~\cite{G00}. 
	Take $k$ as the smallest integer with $k > \frac{2-r}{2(1-r)}$ and use the inequality $e^x \leq 1+x + \ldots +\frac{x^{2k}}{(2k)!}e^{\max\{ x, 0\}}$ to get 
	\begin{equation}\label{develop}
		 \E \left[\exp\left\{   \lambda x_n \hat X\right\}\right]  \leq  1+ \sum_{j=1}^{2k-1} \frac{\lambda ^jx_n^j}{j!}  \E \left[\hat{X}^j \right]	
		 	+ \frac{\lambda^{2k}x_n^{2k}}{(2k)!}  \E \left[ \hat X^{2k}\exp\left\{  \lambda x_n \max\{\hat X, 0\}\right\} \right].
	\end{equation}
	Since $X$ is centred, $\E [\hat{X}] \leq 0$. Due to \eqref{momentass}, the moments $E[\hat{X}^j]$ for $j \leq 2k$ are 
bounded by some constant $C_j$.
The sum can be bounded via
	\begin{align*}
		 \sum_{j=1}^{2k-1}  \frac{\lambda^jx_n^j}{j!}  \E \left[\hat{X}^j \right]  \leq &  \sum_{j=2}^{2k-1}  \frac{\lambda^jx_n^j}{j!}  \E \left[\hat{X}^j \right]  \leq  \frac{\lambda^2x_n^2}{2}  \E \left[\hat{X}^2 \right] + O\left(\frac{1}{  n^{3(1/r -1)}}\right)\\
		 	& =  \frac{\lambda^2x_n^2}{2}+ o\left(\frac{1}{  n^{2(1/r -1)}}\right).
	\end{align*} 
	To treat the last term in \eqref{develop}, we first note that
	the integral 
	\begin{equation*}
		 \E \left[ \hat{X}^{2k}\exp\left\{   \lambda x_n\max\{\hat X, 0\} \right\} \1_{\{ X < 0 \}}\right]
	\end{equation*}
	remains bounded as $n \to \infty$ and so in the sequel we only treat the expectation over the set $\{ X \geq 0 \}$. 
 It is hence sufficient to show that 
	\begin{equation}\label{remafbound}
		\frac{1}{n^{(1+\eta)/r} } \E \left[ \hat X^{2k}\exp\left(  \lambda x_n \frac{\hat X}{\alpha^{1-r} n^{1/r -1}} \right) \1_{\{ X \geq 0 \}}\right] = o\left(\frac{1}{  n^{2(1/r-1)}}\right),
	\end{equation}
	where $\eta  = 2k(1-r)-1>(1-r)>0$. We will use the following inequality for $ K = \delta n^{1/r}$, $\varphi(s) = s^{2k}\exp\left\{   \lambda x_n s\right\} $ with  $\varphi(0)=0$,
	\begin{equation*}
		\E [\varphi(X) \1_{\{K>X>0\}}] \leq \int_0^K \varphi'(s) \P[X > s] \: ds.
	\end{equation*}
We have
\begin{multline}\label{splithave} 
\frac{1}{n^{(1+\eta)/r}} \int_0^K \varphi'(s) \P[ X > s]ds  \\ 
\leq  \frac{1}{n^{(1+\eta)/r}}  \int_0^{\delta n^{1/r}} \lambda x_n s^{2k} e^{\lambda x_n s} \P[ X > s] ds \\+ \frac{1}{n^{(1+\eta)/r}} \int_0^{\delta n^{1/r}} 2k s^{2k-1} e^{\lambda x_n s} \P[ X > s] ds.  
\end{multline}
For the first term on the r.h.s. of \eqref{splithave}, we have
	\begin{align*}
		&  \frac{1}{n^{(1+\eta)/r}}  \int_0^{\delta n^{1/r}} \lambda x_n s^{2k} \exp\left\{   \lambda x_n s \right\} \P[ X > s] ds \\
		& = \frac{1}{n^{(1+\eta)/r}}   \int_0^{\delta n^{1/r}}  \lambda x_n a(s) s^{2k}\exp\left\{   \lambda x_n s -\lambda s^r \right\}  ds  \\
		& = x_n\delta  \frac{1}{n^{\eta/r}}\int_{0}^1 a(\delta n^{1/r} s ) \lambda \delta^{2k} n^{2k/r} s^{2k}\exp\left\{   \lambda n\delta s\left( n^{1/r -1} x_n  -  \delta^{r-1} s^{r-1} \right\} \right)  ds \\
	&\leq  {\rm const}\cdot\frac{1}{n^{\eta/r}}n^{2k/r}  \int_{0}^1 s^{2k}\exp\left\{   \lambda n\delta s\left( n^{1/r -1} x_n  -  \delta^{r-1} s^{r-1} \right) \right\}  ds\, .	
	\end{align*}
	The exponent present in the integral is negative for sufficiently large $n$, since $n^{1/r-1}x_n \to \alpha^{r-1}< \delta^{r-1}\leq \delta^{r-1} s^{r-1}$ for $s \in (0,1]$. To see that the above expression is $o\left(\frac{1}{  n^{2(1/r-1)}}\right)$ take $\varepsilon \in \left(0, \frac{1}{2k+1}\right)$ s.t. $\varepsilon < \frac{1}{r(1+r)}$ and write the integral as a sum of integrals over 
	$(0, n^{-1/r +\varepsilon}]$, $ (n^{- 1/r +\varepsilon}, n^{-1+r^2\varepsilon})$ and $[n^{-1+r^2\varepsilon}, 1)$. The first one is bounded via
	\begin{align*}
		&  \frac{1}{n^{\eta/r}}n^{2k/r}   \int_{0}^{n^{-1/r+\varepsilon}} s^{2k}\exp\left\{\lambda n\delta s\left( n^{1/r -1} x_n  -  \delta^{r-1} s^{r-1} \right) \right\}  ds \\
		& \leq {\rm const} \cdot n^{1-2/r-\eta/r+(2k+1)\varepsilon} \leq  {\rm const} \cdot n^{2(1-1/r)-\eta/r} = o\left(\frac{1}{  n^{2(\frac 1r-1)}}\right).
	\end{align*}
	The integral over the second interval has the following estimate
	\begin{align*}
		&  \frac{1}{n^{\eta/r}}n^{2k/r}\int_{n^{-1/r+\varepsilon}}^{n^{-1+r^2\varepsilon}} s^{2k}\exp\left\{   \lambda n\delta s\left( n^{1/r -1} x_n  -  \delta^{r-1} s^{r-1} \right) \right\}  ds \\
		& \leq  {\rm const} \cdot n^{2k/r}\exp ( 2\lambda\delta n^{r^2\varepsilon} - \lambda\delta^rn^{r\varepsilon}) =  o\left(\frac{1}{  n^{2(1/r-1)}}\right).
	\end{align*}
	The last part can be bounded by
	\begin{align*}
		&   \frac{1}{n^{\eta/r}} n^{2k/r}\int_{n^{-1+r^2\varepsilon}}^1  s^{2k}\exp\left\{   \lambda n\delta s\left( n^{1/r -1} x_n  -  \delta^{r-1} s^{r-1} \right) \right\}  ds \\
		& \leq  {\rm const} \cdot n^{2k/r}\exp \{ \lambda\delta n^{r^2\varepsilon } ( \alpha^{r-1} - \delta^{r-1} + o(1) ) \} =  o\left(\frac{1}{  n^{2(1/r-1)}}\right).
	\end{align*}
       The second term on the r.h.s. of \eqref{splithave} is treated in the same way.
	This proves \eqref{remafbound}
and concludes the proof of the lemma. 
\end{proof}

\begin{proof}[Proof of Lemma  \ref{sumstat}]
Put 
		\begin{equation*}
			q_n =  \frac{(1+\varepsilon) r \log m}{2 \alpha} n^{2-1/r} \quad \mbox{and} \quad t_n = \alpha n^{1/r} + q_n
		\end{equation*}
		and consider the following decomposition with $\delta \in \left(\frac{\alpha}{2^{1/r}}, \alpha\right)$, 
		\begin{align*}
			& m^n \P[S_n \geq t_n]   \\
				 =  & m^n \P\left[S_n \geq t_n, \text{ and } \forall k \leq n,  X_k < \delta n^{1/r}\right]\\
				& +  m^n \P\left[S_n \geq t_n, \text{ and } \exists j\neq i \leq n \text{ s.t. } X_j\wedge X_i \geq \delta n^{1/r}\right] \\
				& + m^n \P\left[S_n \geq t_n, \: \exists j \leq n \text{ s.t. } X_j \in \left[ \delta n^{1/r}, \alpha n^{1/r} -  3q_n\right], \text{ and }\forall k \neq j,  X_k < \delta n^{1/r}\right] \\
				&  + m^n \P\left[S_n \geq t_n, \text{ and }  \exists j \leq n \text{ s.t. }
 X_j > \alpha n^{1/r} -  3 q_n, \text{ and }\forall k \neq j,  X_k < \delta n^{1/r}\right]\\
&  = J_1(n) +J_2(n) +J_3(n) +J_4(n).
		\end{align*}
We have to show that all these terms are summable in $n$.
			As before, we write $\hat{X}_k = X_k \1_{\left\{  X_k< \delta n^{1/r} \right\}}$ and $\hat{S}_n = \sum_{k=1}^n\hat{X}_k$.
		Using the Markov inequality and Lemma~\ref{lem:12} we can estimate $J_1(n)$ in the following way
		\begin{align*}
			J_1(n) & = m^n\P\left[\hat{S}_n \geq t_n \right] = m^n\P\left[ \lambda \frac{\alpha^r n +\lambda \alpha^{2(r-1)} n^{3-2/r}}{t_n} \hat{S}_n \geq \lambda \left(\alpha^r n +\lambda \alpha^{2(r-1)}n^{3-2/r}\right) \right] \\
				& \leq \exp \left\{ - \lambda^2 \alpha^{2(r-1)} n^{3-2/r} \right\} \E\left[ \exp \left\{ \lambda \frac{\alpha^r n + \lambda \alpha^{2(r-1)} n^{3-2/r}}{t_n} \hat{S}_n \right\} \right]\\
				& = \exp \left\{ - \lambda^2 \alpha^{2(r-1)} n^{3-2/r} \right\} \E\left[ \exp \left\{   \lambda \frac{\alpha^r n + \lambda \alpha^{2(r-1)} n^{3-2/r}}{t_n} \hat{X}_1 \right\}  \right]^n\\
				& \leq \exp\left\{ - \lambda^2 \alpha^{2(r-1)}n^{3-2/r} \right\} \exp\left\{ n\lambda^2  \frac{(\alpha^r n + \lambda \alpha^{2(r-1)} n^{3-2/r})^2}{2 t_n^2}  + o(n^{3-2/r})  \right\} \\
				&  \leq \exp\left\{ - \lambda^2 \alpha^{2(r-1)} n^{3-2/r} \right\} \exp\left\{ \frac{\lambda^2 \alpha^{2(r-1)}}{2} n^{3-2/r}  + o(n^{3-2/r})  \right\} \\
				& = \exp\left\{ -\frac{\lambda^2 \alpha^{2(r-1)}}{2} n^{3-2/r}  + o(n^{3-2/r})  \right\}.
		\end{align*} 
		Providing a bound for $J_2(n)$ is easy: we have
		\begin{align*}
			J_2(n) & \leq n^2 m^n \P\left[X > \delta n^{1/r}\right]^2 \leq \const\cdot n^2 \exp \{ \lambda (\alpha^r - 2 \delta^r)n\} 
		\end{align*}
		where the exponent on the right hand side is negative due to the choice of $\delta$. The bound for $J_3(n)$ goes along similar lines as the one for $J_1(n)$. Put 
		\begin{equation*}
			 \tilde{S}_{n-1} = \sum_{k=1}^{n-1}\hat{X}_k, \quad p_n =  \frac{(1+\varepsilon) \lambda \alpha^{2(r-1)}}{2}n^{3-2/r} 
		\end{equation*}
		and write
		\begin{align*}
			J_3(n) &\leq m^n n \P \left[ \tilde{S}_{n-1} + X_n > t_n, X_n < \alpha n^{1/r} -3q_n \right] \\
			& = m^n n \P \left[ \lambda \frac{\alpha^r n  +p_n}{t_n} \left(  \tilde{S}_{n-1} + X_n \right) > \lambda \left(\alpha^r n+p_n\right), X_n < \alpha n^{1/r}-3q_n  \right].
		\end{align*}
		Apply the Markov inequality and a bound for the exponential moment of $\tilde{S}_{n-1}$ as we did it for $\hat{S}_n$ to obtain
		\begin{equation*}
			J_3(n) \leq   \exp\left\{ -\varepsilon \frac{\lambda^2 \alpha^{2(r-1)}}{2} n^{3-2/r}  + o(n^{3-2/r})  \right\} \E\left[  \exp \left\{  \lambda \frac{\alpha^r n + p_n}{t_n} X_n \right\}   \1_{\left\{ X_n \leq \alpha n^{1/r} - 3q_n \right\}} \right].
		\end{equation*}
		It remains to show that
			\begin{equation}\label{eq:forJ3}
			\E\left[ \exp\left\{  \lambda \frac{\alpha^r n + p_n}{t_n} \tilde{X}_n \right\} \1_{\left\{ X_n \leq \alpha n^{1/r} - 3q_n \right\}} \right] =  \exp\left\{ o(n^{3-2/r})  \right\}.
		\end{equation}
		To do so, one can employ the final steps of the proof of Lemma~\ref{lem:12}. That is, the integral 
		\begin{equation*}
			\E\left[ \exp\left\{  \lambda \frac{\alpha^r n + p_n}{t_n} X_n \right\} \1_{\{ X_n\leq 0 \}} \right]
		\end{equation*}
		is bounded. To treat the integral corresponding to the positive values of $X_n$ use the formula
		\begin{equation*}
			\E [\psi(X_n)\1_{\{ K\geq X_n>0\}} ] \leq \int_0^K \psi'(s) \P[X_n>s]ds + \psi(0)\P[X_n> 0]
		\end{equation*}
		with $\psi(s) = \exp\left\{  \lambda \frac{\alpha^r n + p_n}{t_n} s\right\} $ and $K = t_n - 3q_n$. Since $\psi(0)\P[X_n> 0]\leq 1$ we will focus on the integral for which we have
		\begin{align*}
			&\int_0^K \psi'(s) \P[X_n>s]ds  \leq {\rm const} \cdot \int_0^{\alpha n^{1/r}-2q_n}a(s) \exp \left\{  \lambda \frac{\alpha^r n + p_n}{t_n} s - \lambda s^r \right\} ds \\
				& \leq{\rm const}\cdot  n^{2/r}  \int_0^1\exp \left\{  \lambda \frac{(\alpha n^{1/r}-2q_n)}{t_n} s( \alpha^r n + p_n -  t_n (\alpha n^{1/r}-2q_n)^{r-1}) \right\} ds.
		\end{align*}
		To check that the integral is of the form $\exp\left\{ o(n^{3-2/r})  \right\}$ for $r \in (\frac{1}{2},1)$ and sufficiently large $n$ use $r p_n = \alpha^{r-1}n^{1-1/r}q_n$  and write
		\begin{align*}
			\alpha^r n + p_n -  t_n (\alpha n^{1/r}-2q_n)^{r-1} &= p_n - 2(1-r)\alpha^{r-1}n^{1-1/r}q_n - \alpha^{r-1}n^{1-1/r}q_n + o\left(n^{3-2/r}\right) \\
				&= p_n(1-r)(1-2r) + o\left(n^{3-2/r}\right).
		\end{align*}	
		A straightforward upper bound for the integral implies~\eqref{eq:forJ3}. To estimate the last remaining term $J_4(n)$ take $N > \frac 4\varepsilon \vee 4$ and write
		\begin{align*}
			J_4(n)  \leq & \sum_{k=-1}^{3N-1}n m^n \P\left[X_n+\tilde{S}_{n-1} \geq t_n, \:  \frac{k}{N} q_n\leq \alpha n^{1/r} -X_n\leq \frac{k+1}{N}q_n\right] \\
				& + nm^n\P\left[X_n+\tilde{S}_{n-1} \geq t_n, \:  X_n\geq \alpha n^{1/r}+ \frac 1N q_n\right].
		\end{align*}
		To treat the sum just note that
		\begin{align*}
			&  \P\left[X_n+\tilde{S}_{n-1} \geq t_n, \:  \frac{k}{N} q_n\leq \alpha n^{1/r} -X_n\leq \frac{k+1}{N}q_n\right] \\
			&\leq  \P\left[X_n \geq \alpha n^{1/r} - \frac{k+1}{N}q_n, \: \tilde{S}_{n-1} \geq \frac {N+k}Nq_n \right]\\
			& = m^{-n} \exp \left\{\lambda^2r^2\alpha^{2(r-1)}(1+\varepsilon)n^{3-2/r} \frac 1{8N^2} (4N(k+1) - (1+\varepsilon)(N+k)^2 ) (1+o(1)) \right\} \\
			& \leq m^{-n} \exp \left\{ \lambda^2r^2\alpha^{2(r-1)}(1+\varepsilon)n^{3-2/r} \frac 1{8N^2} (-(N-k)^2 +N(4-\varepsilon N/2) )(1+o(1))\right\}.
		\end{align*}
		The last term in the decomposition of $J_4(n)$ is also summable since
		\begin{align*}
			 &nm^n\P\left[X_n+\tilde{S}_{n-1} \geq t_n, \:  X_n\geq \alpha n^{1/r}+ \frac 1N q_n\right] \leq  nm^n\P\left[X_n\geq \alpha n^{1/r}+ \frac 1N q_n\right]\\& =  \exp\left\{ -\frac{r^2}{2N} \lambda ^2\alpha^{2(r-1)}(1+\varepsilon)n^{3-2/r}(1+o(1)) \right\}.
		\end{align*}
\end{proof}

\bibliographystyle{amsplain}

\Addresses

\end{document}